\documentclass[a4paper,10pt]{amsart}
 
\usepackage[english]{babel}
\usepackage{dsfont} 
\usepackage{graphicx}
\usepackage{color}
\usepackage{hyperref}
\usepackage{cleveref}
\usepackage{enumerate}
\usepackage{amssymb} 
\usepackage{amsthm} 
\usepackage{xcolor}
\usepackage[foot]{amsaddr}

\allowdisplaybreaks

\theoremstyle{plain}
\newtheorem{thm}{Theorem}[section]
\newtheorem{lem}[thm]{Lemma}
\newtheorem{prop}[thm]{Proposition}
\newtheorem{cor}[thm]{Corollary}

\newtheorem{alg}[thm]{Algorithm}

\theoremstyle{definition}
\newtheorem{dfn}[thm]{Definition}

\newtheorem{exm}[thm]{Example}

\newtheorem{rem}[thm]{Remark}

\newcommand{\nset}{\mathds{N}}

\newcommand{\qset}{\mathds{Q}}
\newcommand{\rset}{\mathds{R}}
\newcommand{\cset}{\mathds{C}}
\newcommand{\pset}{\mathds{P}}

\newcommand{\diag}{\mathrm{diag}\,}
\newcommand{\diff}{\mathrm{d}}
\newcommand{\lin}{\mathrm{lin}\,}
\newcommand{\loc}{\mathrm{loc}}

\newcommand{\rank}{\mathrm{rank}\,}

\newcommand{\supp}{\mathrm{supp}\,}
\newcommand{\und}{\;\wedge\;}
\newcommand{\folgt}{\;\Rightarrow\;}

\newcommand{\inter}{\mathrm{int}\,}

\newcommand{\one}{\mathds{1}}

\newcommand{\cA}{\mathcal{A}}

\newcommand{\cat}{\mathcal{C}}
\newcommand{\cD}{\mathcal{D}}
\newcommand{\cH}{\mathcal{H}}

\newcommand{\cN}{\mathcal{N}}

\newcommand{\cS}{\mathcal{S}}
\newcommand{\cT}{\mathcal{T}}
\newcommand{\cX}{\mathcal{X}}

\newcommand{\cZ}{\mathcal{Z}}

\newcommand{\sA}{\mathsf{A}}

\newcommand{\sN}{\mathsf{N}}

\newcommand{\fA}{\mathfrak{A}}
\newcommand{\fM}{\mathfrak{M}}

\newcommand{\fp}{\mathfrak{p}}

\author{Philipp J.\ di~Dio}

\address{Technische Universit\"at Berlin, Institut f\"ur Mathematik, Stra\ss{}e des 17.\ Juni 136, D-10623 Berlin, Germany}
\email[A1]{didio@tu-berlin.de}

\begin{title}[Derivatives of Moments: Shape and Gaussian Mixture Reconstruction]{The multidimensional truncated Moment Problem: Shape and Gaussian Mixture Reconstruction from Derivatives of Moments}
\end{title}

\begin{document}
\maketitle

\begin{abstract}
In this paper we introduce the theory of derivatives of moments and (moment) functionals to represent moment functionals by Gaussian mixtures, characteristic functions of polytopes, and simple functions of polytopes. We study, among other measures, Gaussian mixtures, their reconstruction from moments and especially the number of Gaussians needed to represent moment functionals. We find that there are moment functionals $L:\rset[x_1,\dots,x_n]_{\leq 2d}\to\rset$ which can be represented by a sum of $\binom{n+2d}{n} - n\cdot \binom{n+d}{n} + \binom{n}{2}$ Gaussians but not less. Hence, for any $d\in\nset$ and $\varepsilon>0$ we find an $n\in\nset$ such that $L$ can be represented by a sum of $(1-\varepsilon)\cdot\binom{n+2d}{n}$ Gaussians but not less. An upper bound is $\binom{n+2d}{n}-1$.
\end{abstract}

\noindent
\textbf{AMS  Subject  Classification (2010)}. 44A60, 14P99, 30E05, 65D32, 35R30.

\noindent
\textbf{Key  words:} truncated moment problem, Carath\'eodory number, measure reconstruction, Gaussian mixture, generalized eigenvalues, shape reconstruction, algebraic statistics, integral representation


\section{Introduction}

Reconstructing measures from moments is a key problem in statistics \cite{pearson94,titter85,martin05,didio18gaussian}, shape reconstruction \cite{balins61,manas68,mathei80,lee82,milanf95,golub99,becker07, gravin12,gravin14,gravin18,kohn18}, pattern recognition \cite{hu62,dai92,chen93,sommer07,ammari19}, financial mathemiatics \cite{anast06,stoyan16}, and many other fields, and attracts increasing attention especially with the growing usage of computer programs and algorithms to handle such problems. But despite of its growing importance and wide range of application, the theoretical knowledge on the problem of reconstructing measures from moments is very small, especially when only finitely many moments are known. For instance, only recently \cite{didio18gaussian} the question of which truncated moment sequences are represented by Gaussian, log-normal, and more general mixtures was fully answered and the first non-trivial bounds on the required number of summands were given.

While derivatives in the context of moments have been used before, surprisingly no unified approach was introduced so far. In the present paper we present the first unified and systematic approach to reconstruct and investigate measures from moments: \emph{derivatives of moments}. In \Cref{sec:derivMomMeas} we define and investigate derivatives of moments and show that the derivative $\partial^\alpha L := (-1)^{|\alpha|}\cdot L\circ\partial^\alpha$ of a (moment) functional $L$ is represented by the distributional derivative $\partial^\alpha\mu$ of a representing measure $\mu$ of $L$. From this treatment it is clear that $\partial^\alpha L$ is an object that is interesting to investigate on its own account and not only because it solves problems and appears (implicitly or explicitly) in proofs and calculations.

In \Cref{sec:applicDerivMom} we use the concept of derivatives of moments to reprove several known results on  reconstructing polytopes and special measures in a unified and efficient way. Proofs formerly presented over several pages now reduce to a few lines and their key arguments become much more apparent. We use these simplified arguments and proofs to extend these results, e.g., we extend the results from polynomial moments
\[s_\alpha\quad :=\quad \int_{\rset^n} x^\alpha~\diff\mu(x)\]
with $x^\alpha=x_1^{\alpha_1} \cdots x_n^{\alpha_n}$ to non-polynomial moments:
\[s_a\quad :=\quad \int_{\rset^n} a(x)~\diff\mu(x),\]
where $a$ is a measurable (differentiable) function. This allows us to formulate results in full generality and $\partial^\alpha L$ can still be easily calculated from $L$.

In \Cref{sec:gaussian} we return to the reconstruction and investigation of (Gaussian) mixtures. Based on derivatives of moments we fully characterize moment sequences from one ($n$-dimensional) Gaussian distribution
$c\cdot\exp\left(-(x-b)^T A (x-b)\right)$
and we determine $b\in\rset^n$ and $A\in\rset^{n\times n}$ from the moments. While this was known before, our simplified arguments and proofs using derivatives of moments enable us to extend this to mixtures, i.e., linear combinations of e.g.\ Gaussian distributions:
\begin{equation}\label{eq:generalMixture}
F(x)\quad :=\quad \sum_{i=1}^k c_i\cdot \exp\left(-(x-b_i)^T A_i (x-b_i)\right)
\end{equation}
with $c_i\in\rset$ ($c_i>0$), $b_i\in\rset^n$ and $A_i\in\rset^{n\times n}$ for all $i=1,\dots,k$. In the one-dimensional case ($n=1$) we give an explicit way to determine the parameters in (\ref{eq:generalMixture}). Simple formulas are gained under the restriction that $A_1, \dots A_k\in\rset$ are all equal: $A_1=\dots=A_k$. But before we allow the possible relaxation to arbitrary $A_1,\dots,A_k\in\rset^{n\times n}$ we examine the number $k$ of mixtures required to represent a moment sequence $s$, i.e., its minimal number, the (\emph{mixture}) \emph{Carath\'eodory number $\cat_\sA^M(s)$}. Based on very recent results on the Carath\'eodory number $\cat_\sA$ (number of Dirac delta measures, i.e., point evaluations) in \cite{rienerOptima,didioConeArXiv,didio17Cara} and especially \cite{didio19HilbertArxiv} we derive new lower bounds and asymptotic limits for the case of mixtures as well. We show that a non-zero (polynomial) function $p$ with finitely many zeros $\cZ(p)$ gives a moment sequence $s$, resp.\ moment functional $L$, which needs as many components in a mixtures representation as there are linearly independent point evaluation located at $\cZ(p)$, see \Cref{thm:mixtureBoundBelow}. As a consequence (\Cref{cor:lowerBoundsExplicit}) we find that there are moment functionals $L:\rset[x_1,\dots,x_n]_{\leq 2d}\to\rset$, which can be represented by a sum of
\begin{equation}\label{eq:lowerBoundsEven}
\binom{n+2d}{n} - n\cdot \binom{n+d}{n} + \binom{n}{2}
\end{equation}
Gaussian distributions but not less. This disproves the belief that allowing arbitrary $A_i\in\rset^{n\times n}$ with $A_i\succ 0$ reduces the number $k$ of components. Finally, (\ref{eq:lowerBoundsEven}) shows that for each $d\in\nset$ and $\varepsilon>0$ there is an $n\in\nset$ and a moment functional $L:\rset[x_1,\dots,x_n]_{\leq 2d}\to\rset$ which can be represented by a sum of
\[(1-\varepsilon)\cdot\binom{n+2d}{n}\]
Gaussian distributions but not less.

\section{Preliminaries}

Let $\cA$ be a (finite dimensional) real vector space of measurable functions on a measurable space $(\cX,\fA)$. Denote by $L:\cA\rightarrow\rset$ a continuous linear functional. If there is a (positive) measure $\mu$ on $(\cX,\fA)$ such that
\begin{equation}\label{eq:momentFunctionalMeasure}
L(a) = \int_\cX a(x)~\diff\mu(x)\qquad \text{for all}\ a\in\cA,
\end{equation}
then $L$ is called a moment functional. If $\cA$ is finite dimensional, it is a truncated moment functional. By $\sA = \{a_1,\dots,a_m\}$ we denote a basis of the $m$-dimensional real vector space $\cA$ and by
\[s_i := L(a_i)\]
the $a_i$-th (or simply $i$-th) moment of $L$ (or $\mu$ for a $\mu$ as in (\ref{eq:momentFunctionalMeasure})). Given a sequence $s = (s_1,\dots,s_m)\in\rset^m$ we define the Riesz functional $L_s$ by setting $L_s(a_i) = s_i$ for all $i=1,\dots,m$ and extending it linearly to $\cA$, i.e., the Riesz functional induces a bijection between moment sequences $s=(s_1,\dots,s_m)$ and moment functionals $L = L_s$. By $\fM_\sA$ we denote the set of all measures on $(\cX,\fA)$ such that all $a\in\cA$ are integrable and by $\fM_\sA(s)$ or $\fM_\sA(L)$ we denote all representing measures of the moment sequence $s$ resp.\ moment functional $L$. Since the polynomials $\rset[x_1,\dots,x_n]_{\leq 2d}$ are of special importance, we denote by
\[\sA_{n,d}: \{x^\alpha \,|\,\alpha\in\nset_0^n \und |\alpha|=\alpha_1 + \dots + \alpha_m \leq d\}\]
the monomial basis, where we have $x^\alpha = x_1^{\alpha_1}\cdots x_n^{\alpha_n}$ with $\alpha=(\alpha_1,\dots,\alpha_n)\in\nset_0^n$. On $\nset_0^n$ we work with the partial order $\alpha=(\alpha_1,\dots,\alpha_n) \leq \beta = (\beta_1,\dots,\beta_n)$ if $\alpha_i \leq \beta_i$ for all $i=1,\dots,n$.

\begin{dfn}
Let $\sA = \{a_1,\dots,a_m\}$ be a basis of the finite dimensional vector space $\cA$ of measurable functions on the measurable space $(\cX,\fA)$. We define $s_\sA$ by
\[s_\sA: \cX\rightarrow\rset^m,\quad x\mapsto s_\sA(x):= \begin{pmatrix}a_1(x)\\ \vdots\\ a_m(x)\end{pmatrix}.\]
\end{dfn}

Of course, $s_\sA(x)$ is the moment sequence of the Dirac $\delta_x$ measure and the corresponding moment functional is the point evaluation $l_x$ with $l_x(a) := a(x)$. By a measure we always mean a positive measure unless it is explicitly denoted as a signed measure.

The fundamental theorem in the theory of truncated moments is the following.

\begin{thm}[Richter Theorem \cite{richte57}]\label{thm:richter}
Let $\sA = \{a_1,\dots,a_m\}$, $m\in\nset$, be finitely many measurable functions on a measurable space $(\cX,\fA)$. Then every moment sequence $s\in\cS_\sA$ resp.\ moment functional $L:\cA\to\rset$ has a $k$-atomic representing measure
\[s = \sum_{i=1}^k c_i\cdot s_\sA(x_i) \qquad\text{resp.}\qquad L=\sum_{i=1}^k c_i\cdot l_{x_i}\]
with $k\leq m$, $c_1,\dots,c_k>0$, and $x_1,\dots,x_k\in\cX$.
\end{thm}

The theorem can also be called \emph{Richter--Rogosinski--Rosenbloom Theorem} \cite{richte57,rogosi58,rosenb52}, see the discussion after Example 20 in \cite{didioConeArXiv} for more details. That every truncated moment sequence has a $k$-atomic representing measure ensures that the Carath\'eodory number $\cat_\sA$ is well-defined.

\begin{dfn}\label{dfn:caraNumber}
Let $\sA = \{a_1,\dots,a_m\}$ be linearly independent measurable functions on a measurable space $(\cX,\fA)$. For $s\in\cS_\sA$ we define the \emph{Carath\'eodory number $\cat_\sA(s)$ of $s$} by
\[\cat_\sA(s) := \min \{k\in\nset_0 \,|\, \exists\mu\in\fM_\sA(s)\ k\text{-atomic}\}.\]
We define the \emph{Carath\'eodory number $\cat_\sA$ of $\cS_\sA$} by
\[\cat_\sA := \max_{s\in\cS_\sA} \cat_\sA(s).\]The same definition holds for moment functionals $L:\cA\rightarrow\rset$.
\end{dfn}

The following theorem turns out to be a convenient tool for proving lower bounds on the Carath\'eodory number $\cat_\sA$.

\begin{thm}[{\cite[Thm.\ 18]{didio17Cara}}]\label{thm:caraLowerZeroSet}
Let $\sA = \{a_1,\dots,a_m\}$ be measurable functions on a measurable space $(\cX,\fA)$, $s\in\cS_\sA$, and $a\in\cA$ with $a\geq 0$ on $\cX$, $\cZ(a) = \{x_1,\dots,x_k\}$ and $L_s(a) = 0$. Then
\[\cat_\sA \quad\geq\quad \cat_\sA(s) \quad=\quad \dim\lin\{s_\sA(x_i) \,|\, i=1,\dots,k\}.\]
\end{thm}

\begin{rem}
 Note that in \Cref{thm:caraLowerZeroSet} it is crucial that the zero set of $a$ is finite: Take $a=0$ and $\cX=\rset^n$ for a simple example where the statement fails when the zero set is not finite. 
\end{rem}

It is well-known that in general not every sequence $s\in\rset^m$ or linear functional $L:\cA\rightarrow\rset$ has a positive representing measure. But of course it always has a signed $k$-atomic representing measure with $k\leq m$.

\begin{lem}[{\cite[Prop.\ 12]{didioConeArXiv}}]\label{lem:signedMeasures}
Let $\sA = \{a_1,\dots,a_m\}$ be a basis of the finite dimensional space $\cA$ of measurable functions on a measurable space $(\cX,\fA)$. There exist points $x_1,\dots,x_m\in\cX$ such that every vector $s\in\rset^m$ has a signed $k$-atomic representing measure $\mu$ with $k\leq m$ and all atoms are from $\{x_1,\dots,x_m\}$,  i.e., every functional $L:\cA\rightarrow\rset$ is the linear combination $L = c_1 l_{x_1} + \cdots + c_m l_{x_m}${, $c_i\in\rset$}.
\end{lem}

It is well-known that in dimension $n=1$ the atom positions $x_i$ of a moment sequence can be calculated from the generalized eigenvalue problem, see e.g.\ \cite{golub99}.
To formulate this and other results we introduce the following shift.

\begin{dfn}\label{dfn:multipl}
Let $n,d\in\nset$ and $s = (s_\alpha)_{\alpha\in\nset_0^n:|\alpha|\leq d}$. For $\beta\in\nset_0^n$ with $|\beta|\leq d$ we define $M_\beta s := (M_\beta s_\alpha)_{\alpha\in\nset_0^n:|\alpha+\beta|\leq d}$ by $M_\beta s_\alpha := s_{\alpha+\beta}$, i.e., $(M_\beta L)(p) = L(x^\beta\cdot p)$.
\end{dfn}

For a space $\cA$ of measurable functions with basis $\sA = \{a_1,a_2\dots\}$ the \emph{Hankel matrix} $\cH_d(L)$ of a linear functional $L:\cA^2\rightarrow\rset$ is given by $\cH_d(L) = (L(a_i a_j))_{i,j=1}^d$. The atom positions of a truncated moment sequence $s$ (resp.\ moment functional $L$) are then determined from results in \Cref{sec:atomicRecon}.

We use the following notation.

\begin{dfn}\label{dfn:matrixlimit}
Let $s=(s_\alpha)_{\alpha\in\sN},t=(t_\alpha)_{\alpha\in\sN},\dots,z=(z_\alpha)_{\alpha\in\sN}$ be multi-indexed sequences $\alpha\in\sN\subseteq\nset_0^n$ and $l\in\nset$. We define the matrix
\[(s,t,\dots,z)_l := (s_\alpha,t_\alpha,\dots,z_\alpha)_{\alpha\in\sN: |\alpha|\leq {\color{purple}l}}\]
\end{dfn}

For (Gaussian) mixtures we use the following general setting as in \cite{didio18gaussian}:

\begin{dfn}\label{dfn:prob}
Let $\Sigma$ be some fixed set of parameters (in a larger metric space). For all $\sigma\in\Sigma$ and $\xi\in\cX$ we let $\delta_{\sigma,\xi}$ denote probability measures on the measurable space $(\cX,\fA)$ such that:
\begin{enumerate}[i)]
\item All $a\in\cA$ are $\delta_{\sigma,\xi}$-measurable for all $(\sigma,\xi)\in\Sigma\times\cX$, i.e.,
\[\int_\cX |a(x)|~\diff\delta_{\sigma,\xi}(x) \quad<\quad\infty.\]

\item There exists a (unique) $\sigma_0\in\overline{\Sigma}$ (closure of $\Sigma$) such that
\[\int_\cX a(x)~\diff\delta_{\sigma,\xi} \quad\xrightarrow{\sigma\to\sigma_0}\quad \int_\cX a(x)~\diff\delta_\xi(x)\quad =\quad a(\xi)\]
for all $a\in\cA$ and $\xi\in\cX$.
\end{enumerate}
\end{dfn}

For $k\in\nset$ and $\sigma_1,\dots,\sigma_k\in\Sigma$, and $\xi_1,\dots,\xi_k\in\cX$ a ($k$-)\emph{mixture} is then
\[\sum_{i=1}^k c_i\cdot\delta_{\sigma_i,\xi_i}\]
where $\delta_{\sigma_i,\xi_i}$ is its $i$-th \emph{component}. We have $c_i \geq 0$ unless we explicitly speak of \emph{signed} mixtures ($c_i\in\rset$).

Examples $(\delta_{\sigma,\xi},\cX,\Sigma,\sigma_0)$ of this general setting are Gaussian and log-normal measures, see \cite{didio18gaussian}. There we already treated the Carath\'eodory number $\cat_\sA^M$ of mixtures and answered which moment sequences can be represented by mixtures.

\begin{dfn}
If $s\in\cS_\sA$ has a mixture representation, then we define its \emph{(mixture) Carath\'eodory number $\cat^M_\sA(s)$} by
\[\cat^M_\sA(s) := \{k\in\nset_0\,|\, s\ \text{has a mixture representation with}\ k\ \text{components}\}.\]
We call $\cT_\sA\subseteq\cS_\sA$ the \emph{mixture cone}, i.e., the set of all moment sequences which have a (finite) mixture representation. The \emph{(mixture) Carath\'eodory number $\cat_\sA^M$} is then defined by
\[\cat_\sA^M\quad :=\quad \max_{s\in\cT_\sA} \cat^M_\sA(s).\]
\end{dfn}

Of course, since we always have $\cat_\sA^M \leq \dim\cA$, $\cat_\sA^M$ is well-defined. In \cite{didio18gaussian} we gave upper bounds on $\cat_\sA^M$.

\begin{thm}[{\cite[Thm.\ 17(ii)]{didio18gaussian}}]\label{thm:interMixtures}
Let $\cA$ be a finite-dimensional space of continuous functions and $\delta_{\sigma,\xi}$ probability measures as in \Cref{dfn:prob}. Then
\[\inter\cT_\sA\quad =\quad \inter\cS_\sA.\]
\end{thm}

More on the (truncated) moment problem can be found e.g.\ in \cite{stielt94,shohat43,akhiezClassical,kreinMarkovMomentProblem,kemper68,kemper87,landauMomAMSProc, marshallPosPoly,lauren09,fialkow10,lasserreSemiAlgOpt,schmudMomentBook} and references therein.

\section{Reconstruction of atomic Measures}
\label{sec:atomicRecon}

For one-dimensional moment sequences the atom positions of an atomic representing measure can be determined by the following to results.

\begin{lem}\label{lem:genEigenvalueOneDim}
Let $n,d\in\nset$, $\cX = \cset^n$, and $s = (s_0,s_1,\dots,s_{2d+1})\in\rset^{2d+2}$ with
\[s = \sum_{i=1}^k c_i\cdot s_{\sA_{1,2d+1}}(z_i)\]
for some $z_i\in\cset$, $c_i\in\cset$, and $k\leq d$. Then the $z_i$ are unique and are the eigenvalues of the \emph{generalized eigenvalue problem}
\begin{equation}\label{eq:genEigValProb}
\cH_{d}(M_1 s)v_i = z_i \cH_{d}(s)v_i.
\end{equation}
\end{lem}
\begin{proof}
That the $z_i$ are the eigenvalues of (\ref{eq:genEigValProb}) and therefore uniqueness follows from
\[\cH_{d}(s) = (s_{\sA_{1,d}}(z_1),\dots,s_{\sA_{1,d}}(z_k))\cdot\diag(c_1,\dots,c_k)\cdot (s_{\sA_{1,d}}(z_1),\dots,s_{\sA_{1,d}}(z_k))^T\]
and
\begin{multline*}
\cH_{d}(M_1 s) =\\ (s_{\sA_{1,d}}(z_1),\dots,s_{\sA_{1,d}}(z_k))\cdot\diag(c_1 z_1,\dots,c_k z_k)\cdot (s_{\sA_{1,d}}(z_1),\dots,s_{\sA_{1,d}}(z_k))^T.\qedhere
\end{multline*}
\end{proof}

\begin{lem}\label{lem:zerosSupport}
Let $s=(s_0,\dots,s_{2k})\in\rset^{2k+1}$ be a sequence with $\cH(s)\succeq 0$ and $\cH(s)$ is singular (with kernel dimension one). Let $v=(v_k,\dots,v_1,v_0)$ with $\ker\cH(s) = v\cdot\rset$. Then $s$ has a $k$-atomic representing measure $\mu = \sum_{i=1}^k c_i\cdot\delta_{x_i}$ with
\[\{x_1,\dots,x_k\} = \cZ(p) \qquad\text{and}\qquad p(x) = v_k + v_{k-1} x + \dots + v_0 x^k.\]
\end{lem}
\begin{proof}
$L_s(p^2) = v^T\cH(s)v = 0$, i.e., $\supp\mu\subseteq\cZ(p^2)=\cZ(p)$. Equality holds since we work in the one-dimensional framework: $k = |\supp\mu| = \rank\cH(s) = |\cZ(p)|$.
\end{proof}

Compare the preceding results with Vieta's Formulas (\Cref{lem:vieta}).

\section{Derivatives of Moments and Measures}
\label{sec:derivMomMeas}

The following simple and well-known example from the theory of distributions is our motivation in this section. As in the theory of distributions we denote by $\cD(\Omega) = C_0^\infty(\Omega,\rset)$ the set of all test functions and by $\cD'(\Omega)$ the set of all distributions (continuous linear functionals on $\cD(\Omega)$). Most of our applications and examples will work on $\Omega = \cX = \rset^n$, $n\in\nset$.

\begin{exm}\label{exm:simpleMeasureDerivation}
Let $\mu$ on $\cX = \rset$ be given by $\diff\mu := \chi_{[a,b]}\cdot\diff\lambda$, where $\chi_{[a,b]}$ is the characteristic function of the set $[a,b]$, $a < b$, and $\lambda$ is the Lebesgue measure on $\rset$. For $f\in C^1(\rset,\rset)$ we have
\begin{align*}
\mu(\partial_x f) &= \int_\rset \partial_x f~\diff\mu = \int_a^b f'~\diff\lambda = f(b) - f(a) = -\int_\rset f\cdot \partial_x \chi_{[a,b]}~\diff\lambda\\
&=: -\int_\rset f~\diff(\partial_x\mu) = -(\partial_x \mu)(f)
\end{align*}
where we understand $\partial_x \chi_{[a,b]}$ in the distributional sense \cite{grubbDistributions} and  $\partial_x\mu=\delta_a - \delta_b$ as defined above.
\end{exm}

\subsection*{Derivatives of Moments}

Distribution theory motivates the following definition.

\begin{dfn}\label{dfn:MomentDeriv}
Let $\cA$ be a (finite dimensional) vector space of measurable functions, $L:\cA\rightarrow\rset$ be a linear functional, and $\alpha\in\nset_0^n$. If $\partial^\alpha a_i\in\cA$ for some $a_i\in\cA$ we define the \emph{$\alpha$-th derivative $\partial^\alpha s_i$ of $s_i = L(a_i)$} by
\begin{equation}\label{eq:partialsiDef}
\partial^\alpha s_i\quad :=\quad (-1)^{|\alpha|}\cdot L(\partial^\alpha a_i).
\end{equation}
Let $\sA = \{a_1,\dots,a_m\}$, $m\in\nset\cup\{\infty\}$, be a basis of $\cA$. If $\partial^\alpha\cA\subseteq\cA$, then we define the \emph{$\alpha$-th derivative of the sequence $s = (L(a_1),\dots,L(a_m))$} by
\begin{equation}\label{eq:partialsDef}
\partial^\alpha s\quad :=\quad (\partial^\alpha s_1,\dots,\partial^\alpha s_m)
\end{equation}
or equivalently $\partial^\alpha L$ is defined by
\[(\partial^\alpha L)(a)\quad :=\quad (-1)^{|\alpha|}\cdot L(\partial^\alpha a)\]
for all $a\in\cA$ with $\cA$ finite or infinite dimensional.
\end{dfn}

Since $\partial^\alpha\cA\subseteq\cA$ we can calculate $\partial^\alpha s_i$ directly from $L$.

\begin{lem}\label{lem:momentseqDerivCalc}
If $\partial^\alpha a_i = \sum_{j=1}^k c_j a_j\in{\cA}$, $k\in\nset$, then $\partial^\alpha s_i = (-1)^{|\alpha|}\cdot \sum_{j=1}^k c_j s_j$.
\end{lem}
\begin{proof}
$\partial^\alpha s_i = (-1)^{|\alpha|}\cdot {L}(\partial^\alpha a_i)= (-1)^{|\alpha|}\cdot \sum_{j=1}^k c_j {L}(a_j) = (-1)^{|\alpha|}\cdot \sum_{j=1}^k c_j s_j$.
\end{proof}

This provides us with explicit ways to calculate $\partial^\alpha s$ as the next examples show.

\begin{exm}\label{exm:derivMoments}\
\begin{enumerate}[a)]
\item Let $\sA = \sA_{n,d}$ on $\rset^n$, $s=(s_\alpha)\in\cS_{\sA_{n,d}}$, and $\beta\in\nset_0^n$. We have
\[\partial^\beta s = (\partial^\beta s_\alpha) \quad\text{with}\quad \partial^\beta s_\alpha = (-1)^{|\beta|} \mu(\partial^\beta x^\alpha) = \begin{cases}
(-1)^{|\beta|}\cdot\frac{\alpha!}{(\alpha-\beta)!} s_{\alpha-\beta} & \text{if}\ \beta \leq \alpha,\\
0 & \text{else} \end{cases},\]
see also (\ref{eq:momentDerivOneDim}) in \Cref{exm:gaussianDeriv} for $n=1$.\smallskip

\item Let $\sA = \{\exp(d_1 x),\dots,\exp(d_m x)\}$ on $\rset$ with $d_1 < \dots < d_m$ and $s=(s_i)_{i=1}^m\in\cS_\sA$. Then
$\partial_x^k s = (d_1^k\cdot s_1,\dots,d_m^k\cdot s_m)$.\smallskip

\item Let $\sA = \{\sin x, \cos x, \dots, \sin(k x), \cos(k x)\}$ on $\rset$ (or $[0,2\pi)$) for a $k\in\nset$ and $s=(s_1,s_2,\dots,s_{2k-1},s_{2k})\in\cS_\sA$. Then
\begin{align*}
\partial_x s &= (-s_2,s_1,\dots,-k\cdot s_{2k},k\cdot s_{2k-1}),\\
\partial_x^2 s &= (-s_1,-s_2,\dots,-k^2\cdot s_{2k-1},-k^2\cdot s_{2k}),\\
\partial_x^3 s &= (s_2,-s_1,\dots,k^3\cdot s_{2k},-k^3\cdot s_{2k-1}),\\
\partial_x^4 s &= (s_1,s_2,\dots,k^4\cdot s_{2k-1},k^4\cdot s_{2k}),\quad \text{etc.}
\end{align*}
\end{enumerate}
\end{exm}

Note that $\partial^\alpha$ and $M_\beta$ in \Cref{dfn:multipl} ``almost'' commute.

\begin{lem}\label{lem:commutions}
For $\sA = \sA_{n,d}$, $s = (s_\gamma)\in\rset^{|\sA_{n,d}|}$ we have
\[\begin{pmatrix}\alpha+\gamma\\ \beta\end{pmatrix} M_\alpha \partial^\beta s_\gamma = \begin{pmatrix}\gamma\\ \beta\end{pmatrix} \partial^\beta M_\alpha s_\gamma\quad \text{for all}\ \alpha,\beta,\gamma\in\nset_0^n: |\alpha+\gamma| \leq d.\]
\end{lem}

\begin{rem}
When $s$ resp.\ $L$ is a moment sequence/functional, then $\partial s$ resp.\ $\partial L$ (or $-\partial s$ resp.\ $-\partial L$) is in general not a moment sequence. Let $s=(1,1,1)$ be the moment sequence of $\delta_1$ with $\sA = \{1,x,x^2\}$, then $\partial s = (0,-1,-2)$, i.e., $(\partial L)(1) = L_{\partial s}(1) = 0$ but $(\partial L)(x) = L_{\partial s}(x)\neq 0$.
\end{rem}

\begin{lem}
Let $\cA$ be a vector space of measurable functions on the measurable space $(\cX,\fA)$, $1\in\cA$, $\alpha\in\nset_0^n$, $\alpha\neq 0$, $\partial^\alpha\cA\subseteq\cA$, and $L:\cA\rightarrow\rset$ a {linear} functional. The following are equivalent:
\begin{enumerate}[i)]
\item $\partial^\alpha L$ is a moment functional.
\item $\partial^\alpha L = 0$.
\end{enumerate}
\end{lem}
\begin{proof}
While (ii) $\Rightarrow$ (i) is clear, for (i) $\Rightarrow$ (ii) let $\mu$ be a representing measure of $\partial^\alpha L$. Then
\[\mu(\cX) = (\partial^\alpha L)(1) = (-1)^{|\alpha|}\cdot L(\partial^\alpha 1) = 0,\]
i.e., $\mu = 0$ and therefore $\partial^\alpha L = 0$.
\end{proof}

\begin{rem}
Let $\cA$ be a vector space of measurable functions on the measurable space $(\cX,\fA)$, $\alpha\in\nset_0^n$, $\alpha\neq 0$ such that $\partial^\alpha\cA\subseteq\cA$. The following are equivalent:
\begin{enumerate}[i)]
\item For every linear functional $L:\cA\rightarrow\rset$ there exists a $K:\cA\rightarrow\rset$ with $\partial^\alpha K = L$.
\item $\partial^\alpha:\cA\rightarrow\cA$ is {injective}.
\end{enumerate}
Indeed, $\partial^\alpha:\cA\rightarrow \cA$ is injective, if and only if the induced endomorphism $$\cA^*\to\cA^*,\, K\mapsto (-1)^{|\alpha|}\cdot (K\circ\partial^\alpha)=\partial^\alpha K$$ of the dual space is surjective. In \Cref{exm:derivMoments} (a) $\partial^\alpha$ ($\alpha\neq 0$) is not injective, in (c) $\partial^k$ is injective, and in (b) $\partial^k$ is injective if and only if $d_i\neq 0$ for all $i=1,\dots,m$.
\end{rem}

\subsection*{Derivatives of Measures}

In \Cref{exm:simpleMeasureDerivation} we have seen that for the specific measure $\mu$ with $\diff\mu = \chi_{[a,b]}~\diff\lambda$ the derivative is $\partial_x\mu = \delta_a - \delta_b$, of course in the distributional sense:
\[(\partial_x\mu)(f)=f(a)-f(b)=(\delta_a-\delta_b)(f)\quad\text{for all}\ f\in C_0^\infty(\rset,\rset).\]
Here we make use of the notation $\mu(f)$ for $\int f~\diff\mu$ from the theory of distributions that comes in very handy. Note that we can even choose $f\in C^\infty(\rset,\rset)$ since $\supp\mu$ is compact and therefore compactness of $\supp f$ can be omitted. For the rest of this section we want to define $\partial^\alpha\mu$ for measures $\mu$, especially $\mu\in\fM_\sA(s)$, if it exists.

\begin{dfn}\label{dfn:measureDeriv}
Let $\cA$ be a (finite dimensional) vector space of measurable functions, $\mu$ a (signed) measure and $\alpha\in\nset_0^n$. Assume that $\partial^\alpha\cA\subseteq\cA$ and there exists a $\nu\in\cD'(\cX)$ such that
\begin{equation}\label{eq:nuDef}
\nu(f) = (-1)^{|\alpha|}\cdot \mu(\partial^\alpha f)\quad\text{for all}\ f\in\cD(\cX).
\end{equation}
If $\nu$ is a (signed) measure such that all $a\in\cA$ are $\nu$-integrable, then we say the \emph{$\alpha$-th derivative $\partial^\alpha\mu$ of $\mu$ exists on $\cA$} and is defined by
\[\partial^\alpha\mu\ :=\ \nu.\]
\end{dfn}

The following statement, which connects \Cref{dfn:MomentDeriv} with \Cref{dfn:measureDeriv}, is the crucial observation of this section. It enables us to apply results from the theory of distributions to derivatives of moment functionals.

\begin{thm}\label{thm:derivMeasureMoments}
Let $\cA$ be a (finite or infinite dimensional) vector space of measurable functions on the measurable space $(\cX,\fA)$, $L:\cA\rightarrow\rset$ be a moment functional with representing measure $\mu$, and $\alpha\in\nset_0^n$ such that $\partial^\alpha\cA\subseteq\cA$. If $\partial^\alpha\mu$ exists on $\cA$, then $\partial^\alpha\mu$ is a (signed) representing measure of $\partial^\alpha L$, i.e.,
\begin{equation}
\mu(a) = s_a \qquad\folgt\qquad (\partial^\alpha\mu)(a) = \partial^\alpha s_a\qquad \text{for all}\ a\in\cA.
\end{equation}
\end{thm}
\begin{proof}
Since $\partial^\alpha\mu(a)$ exists for all $a\in\cA$ we have
\[(\partial^\alpha\mu)(a) \overset{\text{Def.\ \ref{dfn:measureDeriv}}}{=} \nu(a) \overset{(\ref{eq:nuDef})}{=} (-1)^{|\alpha|}\cdot\mu(\partial^\alpha a) \overset{\text{Def.\ \ref{dfn:MomentDeriv}}}{=} \partial^\alpha s_a.\qedhere\]
\end{proof}

\begin{rem}
\Cref{thm:derivMeasureMoments} says that we can compute the derivative of a moment functional $L$ on $\cA$ by taking the derivative of a representing measure $\mu$ (if its derivative exists on $\cA$) and vice versa. In particular, the result does not depend on the choice of the representing measure. 
\end{rem}

\begin{exm}\label{exm:deltaDeriv}
Let $x\in\cX\subset\rset^n$, $\alpha\in\nset_0^n$, $\cA\subset C^{|\alpha|}(\cX,\rset)$, and $\mu = \delta_x$, then $\partial^\alpha\mu = \partial^\alpha\delta_x$ is given by
\[(\partial^{\alpha}\delta_x)(a) := (-1)^{|\alpha|}\cdot \frac{\partial^{|\alpha|} a}{\partial x^\alpha}(x),\,\textrm{ for all }a\in\cA.\]
Hence, $\delta_x$ is an example of a measure whose derivative is no longer a measure. 
\end{exm}

Besides the Dirac measures also measures of the form $f~\diff\lambda^n$ are very important, where $\lambda^n$ is the $n$-dimensional Lebesgue measure and $f$ is a measurable function.

\begin{dfn}[{\cite[Eq.\ (3.2)]{grubbDistributions}}]
Let $f\in L^1_\loc(\cX)$ and $\lambda^n$ the $n$-dimensional Lebesgue measure on $\cX$. We define the distribution $\Lambda_f$ by
\[\Lambda_f(g)\quad :=\quad \int_\cX g(x) f(x) ~\diff\lambda^n(x),\qquad\text{for all}\ g\in\cD(\cX).\]
\end{dfn}

\begin{thm}[{\cite[Eqs.\ (3.15) and (3.21)]{grubbDistributions}}]\label{thm:LamdaDeriv}
Let $\alpha\in\nset_0^n$. Then
\begin{equation}\label{eq:LamdaDeriv}
\partial^\alpha \Lambda_f = \Lambda_{\partial^\alpha f},\qquad\text{for all}\ f\in L^1_\loc(\cX).
\end{equation}
\end{thm}

If $\partial^\alpha\Lambda_f$ exists on $\cA$, then by \Cref{thm:derivMeasureMoments} we have
\[{(}\partial^\alpha \Lambda_{f}{)}(a) = \Lambda_{\partial^\alpha f}(a) = (-1)^{|\alpha|}\cdot\Lambda_f(\partial^\alpha a) = \partial^\alpha s_a.\]

The following example will be most important in the reconstruction of polytopes and simple functions from their moments, see \Cref{sec:applicDerivMom}.

\begin{exm}\label{exm:areaIntegralDerivative}
Let $f:\rset\rightarrow\rset$ be a continuous and piece-wise linear function with compact support. Let $\xi_1 < \dots < \xi_k$ be the points where $f$ is not differentiable. Then $(\Lambda_f)' = \sum_{i=1}^{k-1} c_i\cdot \chi_{(\xi_i,\xi_{i+1}]}$ and $(\Lambda_f)'' = \sum_{i=1}^k (c_i-c_{i-1})\cdot \delta_{\xi_i}$ where $c_i = \frac{f(\xi_{i+1})-f(\xi_i)}{\xi_{i+1}-\xi_i}$ for $i=1,\ldots,k-1$ and $c_0=c_k=0$ are the slopes of $f$. In particular, $(\Lambda_f)''$ is a signed $k$-atomic measure.
\end{exm}

\begin{exm}\label{exm:rectangleDeriv}
Let $p_{i,j}\in\rset$ be points, $i=1,\dots,n$ and $j=0,1$.
We define the $n$-dimensional hyperrectangle $\square_p$ of $p=(p_{1,0},p_{1,1},\dots,p_{n,0},p_{n,1})\in\rset^{2n}$ by
\[\square_p := [p_{1,0},p_{1,1}]\times\dots\times [p_{n,0},p_{n,1}]\subset\rset^n.\]
The vertices of $\square_p$ are $p_J = (p_{1,j_1},\dots,p_{n,j_n})$ for all $J = (j_1,\dots,j_n)\in \{0,1\}^n$. Since $\square_p$ is compact all moments
\[s_\alpha := \int_{\rset^n} x^\alpha\cdot \chi_p(x)~\diff\lambda^n(x) = \int_{\rset^n} x^\alpha~\diff\Lambda_{\chi_p}(x)\]
for $\alpha\in\nset^n_0$ exist. Here we abbreviated the characteristic function $\chi_{\square_p}$ of $\square_p$ as $\chi_{p}$. Set $\one := (1,\dots,1)$. From the Definitions \ref{dfn:MomentDeriv} and \ref{dfn:measureDeriv} as well as \Cref{thm:LamdaDeriv} we find that
\begin{equation}
\partial^\one s\quad\text{with}\quad \partial^\one s_\alpha = \begin{cases} (-1)^n\cdot \alpha_1\cdots\alpha_n\cdot s_{\alpha-\one} & \text{for}\ \one\leq \alpha,\\ 0 & \text{else}\end{cases} \quad\forall\alpha\in\nset_0^n
\end{equation}
has the signed representing measure
\begin{equation}\label{eq:rectangleMeasureDeriv}
\partial^\one \Lambda_{\chi_p} = \sum_{J\in\{0,1\}^n} (-1)^{|J|} \cdot \delta_{p_J}
\end{equation}
supported only at the vertices $p_J$ of $\square_p$ where $|J| = j_1+\dots+j_n$.
\end{exm}

Gaussian distributions will be considered in \Cref{sec:gaussian}.

\begin{exm}\label{exm:gaussianDeriv}
For $\Lambda_f$ with $f(x) = \exp(-|ax-b|^k)$ all $i$-th moments
\[s_i:=\int_\rset x^i~\diff\Lambda_f(x)\quad\forall i\in\nset_0\]
exist where $a\in\rset_{>0}$, $b\in\rset$, and $k > 0$. For $l\in\nset_0$ we find from the Definitions \ref{dfn:MomentDeriv} and \ref{dfn:measureDeriv} as well as \Cref{thm:LamdaDeriv} that
\begin{equation}\label{eq:momentDerivOneDim}
\partial^l s \quad\text{with}\quad \partial^l s_i = \begin{cases} 0 & \text{for}\ i=1,\dots,l-1,\\ \frac{(-1)^l\cdot i!}{(i-l)!}\cdot s_{i-l} & \text{for}\ i=l,l+1,\dots \end{cases}
\end{equation}
has a signed representing measure given by
\[\Lambda_{\partial^l f} \quad\text{with}\quad \partial^l f = h_l\cdot f\]
for suitable polynomials $h_l$. For $k=2$ we have $h_l(x) = (-a)^l\cdot H_l(ax-b)$ where $H_l$ is the $l$-th Hermite polynomial:
\[H_l(x) = (-1)^l\cdot l! \cdot \sum_{l_1+2l_2=l} \frac{(-1)^{l_1+l_2}}{l_1!\cdot l_2!} (2x)^{l_1}.\]
\end{exm}

\section{Applications}
\label{sec:applicDerivMom}

\subsection*{Polytope Reconstruction}
The problem of reconstructing a (convex {and full-dimensional}) polytope $P\subset\rset^n$, i.e., finding all vertices, is an extensively studied question and several algorithms have been proposed, see e.g.\ \cite{balins61,manas68,mathei80,lee82,milanf95,golub99,becker07, gravin12,gravin14,gravin18,kohn18}, and references therein.

Based on derivatives of moments we will present a simple proof of one version of these algorithms which calculates the vertices from finitely many moments
\[s_\alpha = \int x^\alpha\cdot \chi_P~\diff\lambda^n(x).\]
We use the Brion--Lawrence--Khovanskii--Pukhlikov--Barvinok (BBaKLP) formulas \cite{brion88,lawrence91,barvin91,pukhli92,barvin92} and the generalized eigenvalue problem (as in \Cref{lem:genEigenvalueOneDim}). The aim is to convince the reader that derivatives of moments are a convenient tool for proving and extending the statement in a concise and conceptual way.

Let us state the BBaKLP formulas. This presentation is taken from \cite{gravin12}. Let $P$ be a polytope in $\rset^n$ with vertices $v_1,\dots,v_k$ ($k\geq n+1$), then
\begin{equation}\label{eq:bbaklpFormulaOne}
0 \quad = \quad \sum_{i=1}^k \langle v_i,r\rangle^j \tilde{D}_{v_i}(r) \qquad\text{for all}\ j=0,\dots,n-1,
\end{equation}
see \cite[Eq.\ (3)]{gravin12}, and for $j=n,n+1,\dots$ we have
\begin{equation}\label{eq:bbaklpFormulaTwo}
\int_P \langle x,r\rangle^j~\diff\lambda^n(x)\quad =:\quad s_j(r)\quad =\quad \frac{j!(-1)^n}{(j+n)!}\sum_{i=1}^k \langle v_i,r\rangle^{j+n} \tilde{D}_{v_i}(r), 
\end{equation}
see \cite[Eq.\ (4)]{gravin12}, where $\tilde{D}_{v_i}(r)$ is a rational function on $r\in\rset^n$, i.e., $r$ can be chosen in general position such that $\tilde{D}_{v_i}(\,\cdot\,)$ has no zero or pole at $r$. The $s_j(r)$ is the \emph{$j$-th directional moment} with direction $r$.

\begin{dfn}\label{dfn:areaFunction}
Let $k,n\in\nset$, $P$ be a polytope with vertices $v_1,\dots,v_k\in\rset^n$, $r\in\rset^n\setminus\{0\}$ a vector (of length 1), $a\in\rset$, and $H_{r,a} := \{x\in\rset^n \,|\, \langle r,x\rangle=a\}$ be an affine hyperplane with normal vector $r$. We define the \emph{area function} $\Theta_{P,r}$ to be the $(n-1)$-dimensional volume of $P\cap H_{r,x}$
\[\Theta_{P,r}:\rset\rightarrow\rset,\ x\mapsto\Theta_{P,r}(x) := \mathrm{vol}_{n-1}(P\cap H_{r,x}) = \int_{H_{r,x}} \chi_P(y)~\diff\lambda^{n-1}(y)\]
where $\lambda^{n-1}$ is the $(n-1)$-dimensional Lebesgue measure on $H_{r,x}$.
\end{dfn}

Of course, the area function is integration by parts
\[s_j(r)\quad =\quad \int_{\rset^n} \langle x,r\rangle^j\cdot \chi_P~\diff\lambda^n(x)\quad =\quad \int_{\rset} y^j\cdot\Theta_{P,r}(y)~\diff\lambda(y).\]
The area function $\Theta_{P,r}$ is a continuous piecewise polynomial function of degree $n$ if $r$ is not a normal vector of any facet of $P$. \Cref{exm:areaIntegralDerivative} motivates the following lemma which is the only step where we need the BBaKLP formulas.

\begin{lem}\label{lem:areaFunctionDeriv}
Let $r\in\rset^n$ be a vector of unit length such that $\tilde{D}_{v_i}(r)$ is non-zero and well-defined, i.e., its numerator and denominator is non-zero. Then
\begin{equation}\label{eq:theraDerivBBaKLP}
\partial^n \Lambda_{\Theta_{P,r}}\quad =\quad \sum_{i=1}^k \tilde{D}_{v_i}(r)\cdot\delta_{\langle r,v_i\rangle}.
\end{equation}
\end{lem}
\begin{proof}
Set $y:=\langle x,r\rangle$. From (\ref{eq:bbaklpFormulaOne}) for $j=0,\dots,n-1$ we have
\begin{align*}
\int y^j\cdot \partial^n\Theta_{P,r}(y)~\diff y \overset{(*)}{=} (-1)^n\int \partial^n y^j\cdot \Theta_{P,r}(y)~\diff y = 0 = \sum_{i=1}^k \langle v_i,r\rangle^j \tilde{D}_{v_i}(r)
\end{align*}
and from (\ref{eq:bbaklpFormulaTwo}) with $j'\geq 0$ we have
\begin{align*}
\int y^{n+j'}\cdot \partial^n\Theta_{P,r}(y)~\diff y
&\overset{(+)}{=} (-1)^n\int \partial^n y^{n+j'}\cdot \Theta_{P,r}(y)~\diff y\\
&= \frac{(-1)^n (n+j')!}{j'!} \int y^{j'}\cdot \Theta_{P,r}(y)~\diff y
= \sum_{i=1}^k \langle v_i,r\rangle^{j'+n} \tilde{D}_{v_i}(r).
\end{align*}
Here $(*)$ and $(+)$ hold since $\supp\Theta_{P,r}$ is compact. Thus the claim follows since the set of polynomial functions on a compact set $K$ is dense in $C^\infty(K)$.
\end{proof}

In the previous proof the BBaKLP formulas were used for all monomials $y^j$ ($j\in\nset_0$) and the Weiserstra\ss{} Theorem gives the assertion. But the proof of the lemma can be weakened to the M\"untz--Sz\'asz Theorem \cite{muntz14,szasz16}, i.e., only monomials $\{y^{d_i}\}_{i\in\nset}$ with $\sum_{i\in\in\nset} \frac{1}{d_i} = \infty$ (and $d_1=0$) are necessary. Additionally, the BBaKLP formulas hold only for polynomials but the previous lemma applies to all $C^n$-functions. So we have the following.

\begin{thm}\label{thm:derivNonPoly}
Let $\cA$ be a (finite-dimensional) vector space of measurable functions on $\rset$ with basis $\sA = \{a_1,a_2,\dots\}$ such that $\partial\cA\subseteq\cA$, i.e., $\partial^d\cA\subseteq\cA$ for all $d\in\nset$. Let $P\subset\rset^n$ be a polytope with vertices $v_1,\dots,v_k$, $k\geq n+1$, $r\in\rset^n$ be such that it is neither a pole nor a zero of any $\tilde{D}_{v_i}(\,\cdot\,)$, and consider the directional moments
\[s_j\quad=\quad s_j(r)\quad :=\quad \int_P a_j(\langle x,r\rangle)~\diff\lambda^n(x).\]
Then $\partial^n s$ has an at most $k$-atomic signed representing measure
\[\partial^n\Lambda_{\Theta_{P,r}}\quad =\quad \sum_{i=1}^k \tilde{D}_{v_i}(r)\cdot\delta_{\langle v_i,r\rangle}\]
supported only at the projections $\langle v_i,r\rangle$ of the vertices $v_i$.
\end{thm}
\begin{proof}
Since $s$ has the representing measure $\Lambda_{\Theta_{P,r}}$, the $\partial^n s$ has the at most $k$-atomic representing (signed) measure $\partial^n\Lambda_{\Theta_{P,r}} = \sum_{i=1}^k \tilde{D}_{v_i}(r)\cdot\delta_{\langle v_i,r\rangle}$ by \Cref{thm:derivMeasureMoments} and \Cref{lem:areaFunctionDeriv}.
\end{proof}

What remains is to extract the positions $\langle v_i,r\rangle$ from $\partial^n s$. If $\cA$ consists of polynomials, the generalized eigenvalue problem in \Cref{lem:genEigenvalueOneDim} can be applied. From this we easily get the following corollary, cf.\ e.g.\ \cite[Main Theorem]{gravin12}. Note that we propose to replace Prony's Method/Vandermonde factorization of finite Hankel matrices by the (numerically more stable) generalized eigenvalue problem (as in \Cref{lem:genEigenvalueOneDim}), see \cite[p.\ 1225]{golub99}. For simplicity we assume uniform distribution on $P$. Polynomial distributions on semi-algebraic sets are treated below.

\begin{cor}\label{cor:polytopeFromMoments}
Let $P\subset\rset^n$ be a polytope with vertices $v_1,\dots,v_k$, $k\geq n+1$ and let $r\in\rset^n$ be such that it is neither a pole nor a zero of any $\tilde{D}_{v_i}(\,\cdot\,)$, and for $j=0,\dots,2k-n+1$ let $s_j = s_j(r)$ be the directional moments
\[s_j\quad =\quad \int_P \langle x,r\rangle^j~\diff\lambda^n(x).\]
Then the projections $\xi_i := \langle v_i,r\rangle$ are the eigenvalues of the generalized eigenvalue problem
\begin{equation}\label{eq:genEigValProbPoly}
\cH_k(M_1\partial^n s)y_i\quad =\quad \xi_i\cdot \cH_k(\partial^n s) y_i.
\end{equation}
\end{cor}
\begin{proof}
As in \Cref{thm:derivNonPoly} $s = (s_i)_{i=0}^{2k+1}$ has the representing measure $\Lambda_{\Theta_{P,r}}$ and $\partial^n s$ has the at most $k$-atomic representing (signed) measure $\partial^n\Lambda_{\Theta_{P,r}} = \sum_{i=1}^k \tilde{D}_{v_i}(r)\cdot\delta_{\langle v_i,r\rangle}$ by \Cref{thm:derivMeasureMoments} and \Cref{lem:areaFunctionDeriv}. By \Cref{lem:genEigenvalueOneDim} the positions $\xi_i = \langle v_i,r\rangle$ are the eigenvalues of the generalized eigenvalue problem (\ref{eq:genEigValProbPoly}).
\end{proof}

\begin{rem}
Besides the simple proof, the method of derivatives of moments has another advantage. Since \Cref{lem:areaFunctionDeriv} holds in the distributional sense, \Cref{thm:derivNonPoly} holds for more general functions $a_i$, especially non-polynomial directional moments like in \Cref{exm:derivMoments}(b) or (c). However, the generalized eigenvalue problem must then be replaced by a suitable method to determine the atoms $\delta_{\xi_i}$ from $\partial^n s$.
\end{rem}

\begin{rem}
In \cite[Eq.\ (5)]{gravin12} a ``scaled vector of moments'' is defined in a similar way as $\partial^n s$ in \Cref{exm:derivMoments}(a). However, the strength of \Cref{thm:derivMeasureMoments}, in particular in combination with \Cref{thm:LamdaDeriv}, has not been used.
\end{rem}

\begin{rem}\label{rem:gridpoly}
With $n+1$ different directions $r$ the vertices can be reconstructed using the previous theorem and $(n+1)(2k-n) + 1$ moments are required. If $k$ is unknown, the previous theorem also determines $k$ if sufficiently many directional moments are given.
\end{rem}

Now we extend \Cref{dfn:areaFunction} to functions $f$:
\begin{equation}\label{eq:ThetaDfnFunctionIntegral}
\Theta_{f,r}(x) := \int_{H_{r,x}} f(y)~\diff\lambda^{n-1}(y),
\end{equation}
i.e., integration by part over $H_{r,x}$.

By linearity of integration and differentiation \Cref{cor:polytopeFromMoments} also detects the vertices $v_{i,j}$, $j=1,\dots,d_i$, of full-dimensional polytopes $P_i\subset\rset^n$, $j=1,\dots,p$, from the moments
\begin{equation}\label{eq:polyMomentsSimple}
s_k(r):=\int_{\rset^n}\langle x,r\rangle^k\cdot\chi(x)~\diff\lambda^n(x)
\end{equation}
of the simple function
\begin{equation}\label{eq:polytopeSimpleFunc}
\chi := \sum_{i=1}^p c_i\cdot \chi_{P_i} \qquad(c_i\in\rset,\ c_i\neq0)
\end{equation}
if the $P_i$ or $c_i$ are in general position. We say that a set $\{P_i\}_{i=1}^p$ of polytopes is \emph{in general position} iff $v_{i,j}\neq v_{i',j'}$ for all $(i,j)\neq (i',j')$. Furthermore, we say that $c_1,\dots,c_p$ are \emph{in general position} iff
\begin{equation}\label{eq:polytopeAtomicMeasure}
\mu=\sum_{i=1}^p \sum_{j=1}^{d_i} c_i\cdot\tilde{D}_{v_{i,j}}(r)\cdot\delta_{\langle v_{i,j},r\rangle}
\end{equation}
has non-zero mass $\mu(\langle v_{i,j}, r\rangle)\neq 0$ for $r\in\rset^n$ in general position, i.e., coefficients in (\ref{eq:polytopeAtomicMeasure}) do not cancel out for vertices $v_{i,j}$ with the same projection $\langle v_{i,j},r\rangle$.

\begin{thm}\label{thm:polySimple}
Let $P_i\subset\rset^n$, $i=1,\dots,p$, be full-dimensional polytopes with vertices $v_{i,j}$, $j=1,\dots,d_i$. Let the vertices $v_{i,j}$ or $c_1,\dots,c_p$ be in general position. Let $d:= d_1+\dots+d_p$. Then for a direction $r\in\rset^n$ in general position the projections $\xi_{i,j}:=\langle r,v_{i,j}\rangle$ are the eigenvalues of the generalized eigenvalue problem
\begin{equation}\label{eq:genEigValProbPolySimple}
\cH_d(M_1\partial^n s)y_{i,j}=\xi_{i,j}\cH_d(\partial^n s) y_{i,j}
\end{equation}
where $s_0,\dots,s_{2d-n+1}$ are the directional moments (\ref{eq:polyMomentsSimple}) of (\ref{eq:polytopeSimpleFunc}).
\end{thm}
\begin{proof}
By linearity of $\partial^n$ and \Cref{lem:areaFunctionDeriv} we have that
\[\partial^n \Lambda_{\Theta_{\xi,r}} = \sum_{i=1}^p c_i\cdot \partial^n\Lambda_{\Theta_{P_i,r}} = \sum_{i=1}^p \sum_{j=1}^{d_i} c_i\cdot\tilde{D}_{v_{i,j}}(r)\cdot\delta_{\langle v_{i,j},r\rangle}\]
is a (signed) representing measure of $\partial^n s$ (\Cref{thm:derivMeasureMoments}). Then $(\partial^n \Lambda_{\Theta_{\xi,r}})(\langle r,v_{i,j}\rangle)$ $\neq 0$ for all $i,j$ since the $v_{i,j}$ or $c_i$ are in general position. Hence the projections $\langle r,v_{i,j}\rangle$ are the eigenvalues of (\ref{eq:genEigValProbPolySimple}) by \Cref{lem:genEigenvalueOneDim}.
\end{proof}

\subsection*{Reconstruction of Simple Functions from Moments}
We want to adapt \Cref{thm:polySimple} to simple functions
\[\chi = \sum_{j=1}^k c_j\cdot\chi_{\square_j}\] 
of hyperrectangles $\square_j$, see \Cref{exm:rectangleDeriv}. Similar to polytopes we say that the hyperrectangles $\square_j$ are in \emph{general position} if no two facets of the $\square_j$'s lie in a common hyperplane. The $c_j$'s are called in \emph{general position} if $\partial_i \Lambda_{\Theta_{\chi,e_i}}$ is an at most $2k$-atomic signed measure supported exactly at $p_{j,i,a}$ ($j=1,\dots,k$, $a=0,1$) and $\partial^\one\Theta_{\chi,r}$ is an at most $k\cdot 2^n$-atomic signed measure supported exactly at all $p_{j,i,a}$ ($j=1,\dots,k$, $i=1,\dots,n$, $a=0,1$). We have the following.

\begin{thm}
Let $k,n\in\nset$ and
\begin{equation}\label{eq:linCarac}
\chi\quad =\quad \sum_{j=1}^k c_j\cdot\chi_{\square_j} \qquad\text{with}\qquad c_j\neq 0
\end{equation}
the simple function of hyperrectangles $\square_j$ with $c_1,\dots,c_k$ or $\square_1,\dots,\square_k$ in general position. Consider the moments
\[s_\alpha\quad :=\quad \int x^\alpha\cdot f(x)~\diff\lambda^n(x) \qquad \alpha\in\nset_0^n.\]
Then for each $i=1,\dots,n$ we have
\[\{p_{1,i,0},p_{1,i,1},\dots,p_{k,i,0},p_{k,i,1}\}\quad =\quad\{\xi_{i,1},\dots,\xi_{i,2k}\},\]
i.e., the vertices of the hyperrectangles $\square_j$ are contained in the grid
\begin{equation}\label{eq:gridCarac}
\{\xi_{1,1},\dots,\xi_{1,2k}\}\times\dots\times\{\xi_{n,1},\dots,\xi_{n,2k}\},
\end{equation}
where the $\xi_{i,j}$ are the $2k$ eigenvalues of the generalized eigenvalue problem
\begin{equation}\label{eq:genEigValProbRectangle}
\cH\Big(M_{e_i}\partial_i (s_{l\cdot e_i})_{l=0}^{4k}\Big)y_j\quad =\quad \xi_{i,j}\cdot\cH\Big(\partial_i (s_{l\cdot e_i})_{l=0}^{4k}\Big)y_j
\end{equation}
\end{thm}
\begin{proof}
$t=(s_{l\cdot e_i})_{l=0}^{4k+1}$ are the moments of the area function $\Theta_{f,e_i}$ which has by assumption jumps exactly at the $p_{j,i,l}$'s, $j=1,\dots,k$, $l\in\{0,1\}$. Hence $\partial_i t$ is represented by a signed atomic measure supported exactly at the $p_{j,i,l}$'s by \Cref{thm:derivMeasureMoments} and the positions are gained from the generalized eigenvalue problem (\Cref{lem:genEigenvalueOneDim})
\end{proof}

\begin{rem}\label{rem:polySimple}
For the grid (\ref{eq:gridCarac}) we can then chose an $r\in\rset^n$ in general position such that $\xi\mapsto\langle\xi,r \rangle$ between grid points $\xi$ and their projection $\langle\xi, r\rangle$ is a bijection. Since the $c_j$'s or $\square_j$'s are in general position we can extract these projections from \Cref{thm:polySimple} and uniquely recover the vertices of all $\square_j$'s. The $c_j$'s can then easily (successively) be calculated from evaluation polynomials and $\partial^\one s$. 
\end{rem}

Compared to \Cref{cor:polytopeFromMoments} and \Cref{thm:polySimple} we no longer have the disadvantage that we need to chose $n+1$ random directions $r_i$. We can choose the directions $e_1,\dots,e_n$ and only $r$ in \Cref{rem:polySimple} needs to be in general direction but can be chosen based on the grid (\ref{eq:gridCarac}) from the $e_i$'s. We need to solve $n$ generalized eigenvalue problems (\ref{eq:genEigValProbRectangle}) of size at most $(2k+1)\times (2k+1)$. The choice of $e_i$ is essential so that we cover $2^{n-1}$ vertices of the same $\square_j$ by $\Theta_{\chi,e_i}$ at once and hence get $n$ small generalized eigenvalue problems. Only when we cut the vertices of $\square_j$ out of the grid (\ref{eq:gridCarac}) we need to go to much higher degrees and have to solve one much larger generalized eigenvalue problem based on \Cref{thm:polySimple}. But better options for cutting the vertices $p_{j,J}$ out of (\ref{eq:gridCarac}) might be possible.

\subsection*{Reconstruction of Measures on Semi-Algebraic Sets}
So far we avoided to deal with non-constant densities on bounded sets. Inspired by the work of F.\ Br\'ehard, M.\ Joldes, and \mbox{J.-B.}\ Lasserre \cite{brehard19unpub} we want to demonstrate how our approach can be applied in this case. This and the previous works \cite{lasserre08}, \cite{henrion14}, \cite{marx18arxiv} from (optimal) control applications of the moment-SOS-hierarchy were pointed out to us by the authors of \cite{brehard19unpub}.

Let $G\subseteq\rset^n$ be a semi-algebraic set and $g\in\rset[x_1,\dots,x_n]$ such that $\partial G\subseteq \cZ(g)$. For $f\in C^\infty(\rset^n,\rset)$ and $\Lambda\in\cD(\rset^n)'$ we have the Leibniz formula $\partial_{i}(f\cdot\Lambda) = \partial_{i}f \cdot \Lambda + f\cdot \partial_{i}\Lambda$ \cite[Lem.\ 3.7]{grubbDistributions} and if $\Lambda = \chi_G$ then $\partial_{i} \chi_G$ acts on test functions $\varphi\in\cD(\rset^n)$ as (weighted) $(n-1)$-dimensional Lebesgue measure supported on $\partial G$, i.e., $(\partial_{i}\chi_G)(g\cdot\varphi)=0$ for all $\varphi\in\cD(\rset^n)$ \cite[p.\ 33]{grubbDistributions}. For $g=\sum_\alpha g_\alpha x^\alpha$ we set $g(M):= \sum_\alpha g_\alpha M_\alpha$, where the $M_\alpha$ are the shifts from \Cref{dfn:multipl}. Remember the matrix notation $(s,t,\dots,z)_l$ from \Cref{dfn:matrixlimit}.

\begin{thm}[{\cite[Thm.\ 1]{brehard19unpub}}]\label{thm:joldes}
Let $G\subseteq\rset$ be a semi-algebraic set, let $g\in\rset[x_1,\dots,x_n]$ with $\gamma:=\deg g$ and $\partial G \subseteq\cZ(g)$, $p\in\rset[x_1,\dots,x_n]$ with $d := \deg p$, and $s_\alpha$ the moments of $\exp(p)\cdot\chi_G$,
\[s_\alpha\quad :=\quad \int_G x^\alpha\cdot \exp(p(x))~\diff\lambda^n(x),\]
for all $\alpha\in\nset_0^n$ with $|\alpha|\leq k$ for some $k \geq 2d + 2\gamma -2$. The following are equivalent:
\begin{enumerate}[i)]
\item $p=\sum_{\alpha\in\nset_0:|\alpha|\leq d} c_\alpha\cdot x^\alpha$.

\item For each $i=1,\dots,n$ let $\alpha^{(1)},\alpha^{(2)},\dots,\alpha^{(m)}$ with $m=\binom{n+d-1}n$ denote an enumeration of $\alpha=(\alpha_1,\dots,\alpha_n)\in\nset_0^n$ with $|\alpha|\leq d$ and $\alpha_i\geq 1$. The kernel of
\begin{equation}\label{eq:lasserre}
(g(M)\partial_{x_i} s,\ g(M)M_{\alpha^{(1)}-e_i}s,\ \dots,\ g(M)M_{\alpha^{(m)}-e_i}s)_{k-d}
\end{equation}
is spanned by $(1,-\alpha_i^{(1)}\cdot c_{\alpha^{(1)}},\dots,-\alpha_i^{(m)}\cdot c_{\alpha^{(m)}})$ for every $i=1,\dots,n$.
\end{enumerate}
$c_0$ is determined by normalization. If $g\geq 0$ on $G$ then $k\geq 2d + \gamma -2$ is sufficient.
\end{thm}
\begin{proof}
Note that $s$ is represented by $\exp(p)\cdot\chi_G$, $\partial_{i} s$ is presented by $\partial_{i}(\exp(p)\cdot\chi_G) = \partial_{i}p\cdot \exp(p)\cdot\chi_G + \exp(p)\cdot \partial_{i}\chi_G$ and since $(\partial_{i}\chi_G)(g\cdot\varphi)=0$ for all $\varphi\in\cD(\rset^n)$ we finally have that $g(M)\partial_{i}s$ is represented by $g\cdot\partial_{i}p\cdot\exp(p)\cdot\chi_G$.

(ii) $\Rightarrow$ (i): So by the previous note $\partial_i p$, i.e., $(1,-\alpha_i^{(1)}\cdot c_{\alpha^{(1)}},\dots,-\alpha_i^{(m)}\cdot c_{\alpha^{(m)}})$, is in the kernel and all $c_\alpha$'s with $\alpha\neq 0$ are determined.

(i) $\Rightarrow$ (ii): Again $\partial_i p$, i.e., $(1,-\alpha_i^{(1)}\cdot c_{\alpha^{(1)}},\dots,-\alpha_i^{(m)}\cdot c_{\alpha^{(m)}})$, is in the kernel of (\ref{eq:lasserre}). It is sufficient to show that $(g(M)M_{\alpha^{(1)}-e_i}s,\ \dots,\ g(M)M_{\alpha^{(m)}-e_i}s)_{k-d}$ is full dimensional to show that the kernel is one-dimensional. Assume the columns of $(g(M)M_{\alpha^{(1)}-e_i}s,\ \dots,\ g(M)M_{\alpha^{(m)}-e_i}s)_{k-d}$ are linearly dependent, then by the linearity of the shift $M_\alpha$ also the columns of 
\[(g^2(M)M_{\alpha^{(1)}-e_i}s,\ \dots,\ g^2(M)M_{\alpha^{(m)}-e_i}s)_{d-1}\tag{$*$}\]
are linearly dependent. But ($*$) is the Hankel matrix of $g^2(M)s$, a moment sequence with representing measure $g^2\cdot\exp(p)\cdot\chi_G$, i.e., has full rank. This proves that the kernel of (\ref{eq:lasserre}) is one-dimensional.

If $g\geq 0$ on $G$, squaring $g$ in ``(ii) $\Rightarrow$ (i)'' is not necessary and linear independence already holds for $k\geq 2d + \gamma -1$.
\end{proof}

The bound $k\geq 2d + 2\gamma -2$, resp.\ $k\geq 2d + \gamma -2$, comes from the maximal $\alpha$, i.e., $s_\alpha$, needed to construct (\ref{eq:lasserre}). If $d = \deg p$ is unknown, then the previous theorem also recovers $d$ if $k$ is large enough. For $k\geq 2d + 2\gamma - 2$ the kernel of (\ref{eq:lasserre}) is one-dimensional, i.e., determines $d$ as $\max_{c_\alpha\neq 0} |\alpha|$. For $k < 2d+2\gamma-2$ (resp.\ $2d+\gamma-2$) (\ref{eq:lasserre}) is full rank.

In \cite{brehard19unpub} also the problem of finding $g$ from $s=(s_\alpha)$ for an unknown $G$ is addressed, but then all moments $s_\alpha$ are necessary.

\section{Gaussian Mixtures}
\label{sec:gaussian}

\subsection*{One component} For a Gaussian distribution $g(x) = c\cdot\exp(-\frac{a}{2}(x-b)^2)$ on $\rset$ we have
\begin{equation}\label{eq:GaussianDeriv}
g'(x)\ =\ -a(x-b)\cdot g(x)\ =\ -ax\cdot g(x) + ab\cdot g(x).
\end{equation}
So integration over $x^i\cdot g'(x)$ gives
\begin{equation}\label{eq:GaussianRelation}
-i\cdot s_{i-1}\ =\ (\partial s)_i\ =\ -a\cdot (M_1 s)_i + ab\cdot s_i\ =\ -a s_{i+1} + ab\cdot s_i,\quad\text{for all}\ i\in\nset_0,
\end{equation}
see also \cite[Eq.\ (5)]{amendo16}. This implies the following result.

\begin{lem}[{\cite[Prop.\ 1]{amendo16}}]\label{lem:gaussianCharacOne}
Let $k\in\nset$, $k\geq 2$, be a natural number and $s=(s_0,s_1,\dots,s_k)$ be a real sequence with $s_0\neq 0$. The following are equivalent:
\begin{enumerate}[i)]
\item $s$ is the moment sequence of the Gaussian distribution $c\cdot \exp(-\frac{a}{2}(x-b)^2)$ with $a,b,c\in\rset$, $a>0$, $c\neq 0$, i.e., $s_i = \int x^i\cdot c\cdot \exp(-\frac{a}{2}(x-b)^2)~\diff x$.

\item There are $a,b\in\rset$ with $a>0$ such that the matrix
\[(\partial s, s, M_1 s)_{k-1} = \begin{pmatrix}
0 & s_0 & s_1\\
-s_0 & s_1 & s_2\\
-2\cdot s_1 & s_2 & s_3\\
\vdots & \vdots & \vdots\\
-(k-1)\cdot s_{k-2} & s_{k-1} & s_k
\end{pmatrix}\]
has rank two with kernel $(1,-ab,a)^T\cdot\rset$.
\end{enumerate}
In this case, one has $a=\frac{s_0^2}{s_0s_2-s_1^2}$, $b=\frac{s_1}{s_0}$ and $c = s_0\cdot \sqrt{\frac{a}{\pi}}$.
\end{lem}
\begin{proof}
While (i) $\Rightarrow$ (ii) is clear, we show (ii) $\Rightarrow$ (i) by induction on $i$. Since $0\neq s_0= c\cdot\int e^{-a(x-b)^2}~\diff x$ for $c = s_0\cdot \sqrt{\frac{a}{\pi}}$ and $s_{-1}:=0$, we have by (ii), (\ref{eq:GaussianDeriv}), (\ref{eq:GaussianRelation}) and the induction hypothesis that
\begin{align*}
a\cdot s_{i+1} &= i\cdot s_{i-1} + ab\cdot s_i\\
&= \int \partial x^i\cdot c\cdot\exp(-a(x-b)^2)~\diff x + \int ab\cdot x^i\cdot c\cdot \exp(-a(x-b)^2)~\diff x\\
&= \int [-x^i (-ax+ab) + ab\cdot x^i]\cdot c\cdot\exp(-a(x-b)^2)~\diff x\\
&= a\cdot\int x^{i+1}\cdot c\cdot \exp(-a(x-b)^2)~\diff x \qquad\text{for all}\ i=0,\dots,k-1,
\end{align*}
i.e., $s_{i+1}$ is the $(i+1)$-th moment of $c \cdot \exp(-a(x-b)^2)$.
\end{proof}

On $\rset^n$ we have the following.

\begin{thm}\label{thm:gaussianCharacMulti}
Let $n\in\nset$, $A = (a_1,\dots,a_n) = (a_{i,j})_{i,j=1}^n\in\rset^{n\times n}$ be a symmetric and positive definite matrix, $b\in\rset^n$, $c\in\rset$, $c\neq 0$, and $k\in\nset$ {with $k\geq 2$}. Set
\[g(x) := c\cdot e^{-\frac{1}{2}(x-b)^T A (x-b)}.\]
For a multi-indexed real sequence $s = (s_\alpha)_{\alpha\in\nset_0^n:|\alpha|\leq k}$ the following are equivalent:
\begin{enumerate}[i)]
\item $s$ is the moment sequence of $\Lambda_g$, i.e., $s_\alpha = \int x^\alpha \cdot g(x)~\diff\lambda^n(x)$ for all $\alpha\in\nset_0^n$ with $|\alpha|\leq k$.
\item For $i=1,\dots, n$ the matrix $(\partial_i s, s, M_{e_1} s,\dots M_{e_n} s)_{k-1}$ has the $1$-dimensional kernel
\begin{equation}\label{eq:gaussianKernelMulti}
(1,-\langle b, a_i\rangle, a_{i,1},\dots,a_{i,n})\cdot\rset.
\end{equation}
\end{enumerate}
\end{thm}
\begin{proof}
For $i=1,\dots,n$ we have
\[0 =
\partial_i g(x) - \langle b,a_i\rangle\cdot g(x) + a_{i,1} x_1\cdot g(x) + \dots + a_{i,n} x_n\cdot g(x).\tag{$*$}\]

(i) $\Rightarrow$ (ii): From ($*$) we find that (\ref{eq:gaussianKernelMulti}) is contained in the kernel of {the matrix $(\partial_i s, s, M_{e_1} s,\dots, M_{e_n}s)_{k-1}$. It suffices to show that the kernel of the matrix $(\partial_i s, s, M_{e_1} s,\dots, M_{e_n} s)_{1}$} is at most one-dimensional. Consider
\[H := \begin{pmatrix}
s_0 & s_{e_1} & \dots & s_{e_n}\\
s_{e_1} & s_{2e_1} & \dots & s_{e_1+e_n}\\
\vdots & \vdots & & \vdots\\
s_{e_n} & s_{e_1+e_n} & \dots & s_{2e_n}
\end{pmatrix},\]
the Hankel matrix of $L_s|_{\rset[x_1,\dots,x_n]_{\leq 2}}$. Let $d=(d_0,\dots,d_n)\in\ker H$. Then $0 = L_s(\langle d, (1,x_1,\dots,x_n)\rangle^2) = \int (d_0 + d_1 x_1 + \dots + d_n x_n)^2~\diff\Lambda_g(x)$ implies $d=0$, i.e., $H$ has full rank $n+1$. Therefore { $(\partial_i s, s, M_{e_1} s,\dots, M_{e_n} s)_1$} has rank at least $n+1$ since it has $H$ as submatrix. Its kernel can thus be at most one-dimensional.

(ii) $\Rightarrow$ (i): Let $O\in\rset^{n\times n}$ be an orthogonal matrix such that $O\cdot A\cdot O^T = \diag(\lambda_1,\dots,\lambda_n)$, $\lambda_i >0$. The coordinate change on $\rset^n$ given by $y=Ox$ induces a linear transformation on the space of moment sequences. Let $t=(t_\alpha)_{|\alpha|\leq k}$ be the moment sequence obtained from $s$ via this transformation. A straight-forward calculation shows that
\begin{align*}
\ker (\partial_i t, t, M_{e_1} t,\dots, M_{e_n}t)_1
&= \ker (\partial_i t, t, M_{e_1} t,\dots, M_{e_n}t)_{k-1}\\
&= (1,-\lambda_i\tilde{b}_i,0,\dots,0,\lambda_i,0,\dots,0)^T\cdot \rset,
\end{align*}
where $\tilde{b}=Ob$. This means that we are in the 1-dimensional setting
\[\ker (\partial_i (t_{j\cdot e_i})_{j=1}^k,(t_{j\cdot e_i})_{j=1}^k,M_{e_i}(t_{j\cdot e_i})_{j=1}^k)=(1,-\lambda_i \tilde{b}_{i},\lambda_i)^T\cdot\rset\]
where the $1$-dimensional assertion holds by \Cref{lem:gaussianCharacOne}. Hence, $t=(t_\beta)$ is represented by $t_0\cdot \frac{\sqrt{\lambda_1\cdots\lambda_n}}{(\pi)^{n/2}} \prod_{i=1}^n e^{-\frac{\lambda_i}{2}(y_i-\tilde{b}_i)^2}$. The inverse transformation $x = O^T y$ together with $\lambda_1\cdots\lambda_n = \det(A)$ gives the $n$-dimensional assertion.
\end{proof}

Hence, the previous theorem provides an easy way to determine $A\in\rset^{n\times n}$ and $b\in\rset^n$ from the moments $s_\alpha$.

\begin{alg}\
\begin{enumerate}[Step 1:]
\item[\textbf{Input:}] $k\in\nset$, $k\geq {2}$; $s = (s_\alpha)_{\alpha\in\nset_0^n: |\alpha|\leq k}$.

\item For $i=1,\dots, n$:
\begin{enumerate}[a)]
\item Calculate $\beta_i$ and $a_i=(a_{i,1},\dots,a_{i,n})$ from
\[\ker(\partial_i s, s, M_{e_1}s,\dots M_{e_n} s)_{1}=(1,-\beta_i, a_{i,1},\dots,a_{i,n})\cdot\rset\]
\begin{itemize}
\item[-] If the kernel is not one-dimensional, then $s$ is not represented by one Gaussian distribution.
\end{itemize}

\item Check: $(1,-\beta_i, a_{i,1},\dots,a_{i,n})\in\ker(\partial_i s, s, M_{e_1} s,\dots M_{e_n} s)_{k-1}$?
\begin{itemize}
\item[-] If FALSE: $s$ is not represented by one Gaussian distribution.
\end{itemize}
\end{enumerate}

\item Check: $A = (a_{i,j})_{i,j=1}^n$ is symmetric and positive definite?
\begin{itemize}
\item[-] If FALSE: $s$ is not represented by one Gaussian distribution.
\end{itemize}

\item Calculate $b = A^{-1}\cdot (\beta_1,\dots,\beta_n)^T$ and $c = \frac{\sqrt{\det(A)}}{\pi^{n/2}}\cdot s_0$.

\item[\textbf{Out:}] ``$s$ is represented by a Gaussian distribution'': TRUE or FALSE. If TRUE: $A$, $b$, $c$.
\end{enumerate}
\end{alg}

With
\begin{equation}\label{eq:otherGaussianDeriv}
h'(x) = -a(x-b)^{2d-1}\cdot h(x) = -a\sum_{i=0}^{2d-1} \begin{pmatrix} 2d-1\\ i\end{pmatrix} x^i\cdot (-b)^{2d-1-i}\cdot h(x)
\end{equation}
we get a result similar to \Cref{thm:joldes} but with integration over $\rset^n$ instead of a semi-algebraic set $G$.

\begin{thm}\label{thm:GaussianCharacHigher}
Let $k,d\in\nset$ with $k\geq 4d-2$ and $s=(s_0,\dots,s_k)$ be a real sequence with $s_0\neq 0$. The following are equivalent:
\begin{enumerate}[i)]
\item $s$ is the moment sequence of the distribution $c\cdot \exp\left(\frac{-a}{2d}(x-b)^{2d}\right)$ with $a,b,c\in\rset$, $a>0$, $c\neq 0$.

\item There are $a,b\in\rset$ with $a>0$ such that the matrix
\begin{multline*}
(\partial_x s, s, M_1 s, \dots, M_{2d-1} s)_{k-2d+1} =\\
\begin{pmatrix}
0 & s_0 & \cdots & s_{2d-1}\\
-s_0 & s_1 & \cdots & s_{2d}\\
\vdots &\vdots & & \vdots\\
-(k-2d+1)\cdot s_{k-2d} & s_{k-2d+1} & \cdots & s_k
\end{pmatrix}
\end{multline*}
has a one-dimensional kernel spanned by
\[(1,a\left(\begin{smallmatrix} 2d-1\\ 0 \end{smallmatrix}\right)(-b)^{2d-1},a\left(\begin{smallmatrix} 2d-1\\ 1 \end{smallmatrix}\right)(-b)^{2d-2},\dots,a\left(\begin{smallmatrix} 2d-1\\ 2d-2 \end{smallmatrix}\right)(-b),a\left(\begin{smallmatrix} 2d-1\\ 2d-1 \end{smallmatrix}\right))^T,\]
and $s_i$ is the $i$-th moment of $c\cdot \exp\left(\frac{-a}{2d}(x-b)^{2d}\right)$ for $i=0,\dots,2d-2$.
\end{enumerate}
In this case $c = \sqrt[2d]{\frac{a}{2d}}\cdot\frac{s_0}{2\cdot\Gamma(1+\frac{1}{2d})}$.
\end{thm}
\begin{proof}
Similar to the proof of \Cref{lem:gaussianCharacOne} using (\ref{eq:otherGaussianDeriv}) instead of (\ref{eq:GaussianRelation}) in the induction. The formula for $c$ follows from $\int_\rset \exp(-x^{2d})~\diff x = 2\cdot\Gamma(1+\frac{1}{2d})$, $d\in\nset$.
\end{proof}

\subsection*{Multiple components in dimension one with same variance.}

While we fully characterized all moment sequences represented by one Gaussian distribution and showed how to determine the parameters, let us investigate mixtures with more than one component. In this study the elementary symmetric polynomials play a crucial role.

\begin{dfn}
For $k,j\in\nset$ with $j\leq k$ we denote by
\[\sigma_l(b_1,\dots,b_k)\quad :=\quad \sum_{1\leq j_1 < j_2 < \dots < j_l\leq k} b_{j_1} b_{j_2}\cdots b_{j_l}\]
the \emph{elementary symmetric polynomials}.
\end{dfn}

The elementary symmetric polynomials have the following property.

\begin{lem}[Vieta's Formulas]\label{lem:vieta}
Let $k\in\nset$ and $b_1,\dots,b_k\in\rset$ be pairwise different points. For $v_1,\dots,v_k\in\rset$ the following are equivalent:
\begin{enumerate}[i)]
\item \[\ker\begin{pmatrix}
1 & b_1 & b_1^2 & \hdots & b_1^k\\
\vdots & \vdots & \vdots & & \vdots \\
1 & b_k & b_k^2 & \hdots & b_k^k
\end{pmatrix}\quad =\quad (v_k,v_{k-1},\dots,v_1,1)\cdot\rset\]
\item $v_l = (-1)^l\cdot \sigma_l(b_1,\dots,b_k)$ for all $l=1,\dots,k$.
\item $\cZ(p) = \{b_1,\dots,b_k\}$ with
\[p(x)\quad =\quad \prod_{j=1}^k (\lambda-b_j)\quad =\quad x^k + v_1 x^{k-1} + v_2 x^{k-2} + \dots + v_k.\]
\end{enumerate}
\end{lem}
\begin{proof}
Follows directly from 
\[p(x)\quad =\quad \prod_{j=1}^k (\lambda-b_j)\quad =\quad x^k - \sigma_1 x^{k-1}  + \sigma_2 x^{k-2} \mp \dots + (-1)^k \sigma_k.\qedhere\]
\end{proof}

Since we assume all Gaussian distributions to have the same variance, we introduce the following convenient operator.

\begin{dfn}\label{dfn:deltaa}
Let $L:\rset[x]_{\leq d}\to\rset$ be a linear functional ($s\in\rset^{d+1}$) with $d\in\nset\cup\{\infty\}$ and a differentiable function $f\in C^1(\rset,\rset)$. For $a\in\rset\setminus\{0\}$ we define
\[\Delta_a L\; :=\; \frac{1}{a}(\partial+aM_1)L \qquad\text{and}\qquad (\Delta_a f)(x)\; :=\; \frac{1}{a}(f'(x) + ax f(x)).\]
\end{dfn}

Note, that we use $\Delta_a$ as an operator acting on functionals and on functions to emphasize the close connection between the operations performed on $L$ and measures $\mu = \Lambda_f$ provided by \Cref{thm:derivMeasureMoments}.

$\Delta_a$ has the following properties (Lemmas \ref{lem:deltaaPropStart}--\ref{lem:deltaaPropEnd}).

\begin{lem}\label{lem:deltaaPropStart}
Let $L:\rset[x]_{\leq d}\to\rset$ be a linear functional with $d\in\nset\cup\{\infty\}$ and $a\neq 0$. Then
\[\Delta_a L = 0\qquad\Rightarrow\qquad L(x^n) = 0\quad\text{for all}\ n=1,2,\dots,d.\]
\end{lem}
\begin{proof}
\underline{$n=1$:} $0 = (\Delta_a L)(1) = \frac{1}{a}((\partial + aM_1)L)(1) = \frac{1}{a} (-L(\partial 1) + a\cdot L(x)) = L(x)$.

\underline{$n\to n+1$:} With $L(x^i)=0$ for all $i=1,\dots,n$ it follows that $0 = (\Delta_a L)(x^n) = \frac{1}{a}((\partial + aM_1)L)(x^n)$ $= \frac{1}{a}(-n\cdot L(x^{n-1}) + a\cdot L(x^{n+1})) = L(x^{n+1})$.
\end{proof}

\begin{lem}\label{lem:knownaMeasDeriv}
Let $k\in\nset$, $a\in\rset\setminus\{0\}$, and
\[F(x)\quad =\quad \sum_{i=1}^k c_i\cdot\exp\left(-\frac{a}{2}(x-b_i)^2\right)\]
for some $b_1,\dots,b_k\in\rset$ pairwise different and $c_1,\dots,c_k\in\rset\setminus\{0\}$. Then
\begin{equation}\label{eq:knownaMeasDeriv}
\Delta_a^l F(x)\quad =\quad \sum_{i=1}^k c_i\cdot b_i^l\cdot \exp\left(-\frac{a}{2}(x-b_i)^2\right)
\end{equation}
for every $l\in\nset_0$.
\end{lem}
\begin{proof}
Follows by induction on $l$. $l=0$ is clear. We have to show $l\to l+1$:
\begin{align*}
\Delta_a \sum_{i=1}^k c_i b_i^l\exp\left(-\frac{a}{2}(x-b_i)^2\right)
&= \frac{1}{a}(\partial + ax)\sum_{i=1}^k c_i b_i^l\exp\left(-\frac{a}{2}(x-b_i)^2\right)\\
&=\frac{1}{a}\sum_{i=1}^k c_i b_i^l (-ax+ab_i +ax)\exp\left(-\frac{a}{2}(x-b_i)^2\right)\\
&= \sum_{i=1}^k c_i b_i^{l+1}\exp\left(-\frac{a}{2}(x-b_i)^2\right).\qedhere
\end{align*}
\end{proof}

\begin{lem}\label{lem:transform}
Let $d\in\nset$, $a>0$, $b\in\rset$, and $\sA=\{1,x,\dots,x^d\}$. Define
\begin{equation}\label{eq:tabDef}
t_a(b)\quad :=\quad \left( \int_\rset x^i\cdot \exp\left(-\frac{a}{2}(x-b)^2\right) \right)_{i=0}^d,
\end{equation}
i.e., $t_a(b)$ is the moment vector of the moments $s_0,\dots,s_d$ of $\exp\left(-\frac{a}{2}(x-b)^2\right)$. Then there is an invertable matrix $\fM=\fM(a)\in\rset^{(d+1)\times(d+1)}$ such that
\[\fM t_a(b)\quad =\quad s_\sA(b)\]
and it follows that
\[\fM\Delta_a^l t_a(b)\; =\; \fM\Delta_a^l \fM^{-1} \fM t_a(b)\; =\; M_l s_\sA(b),\qquad i.e.,\qquad \fM\Delta_a^l\fM^{-1}\; =\; M_l.\]
\end{lem}
\begin{proof}
Since
\[\int_\rset x^i\cdot \exp\left(-\frac{a}{2}(x-b)^2\right)\quad =\quad \int_\rset (x+b)^i\cdot \exp\left(-\frac{a}{2}x^2\right)\]
we find that the $i$-th entry in $t_a(b)$ is a polynomial of degree $i$ in $b$. The coordinate change to $\sA = \{1,b,\dots,b^i\}$ is $\fM$. The second statement follows immediately from
\begin{align*}
\Delta_a^l \int_\rset x^i\cdot\exp\left(-\frac{a}{2}(x-b)^2\right)~\diff x &= \int_\rset x^i\cdot \Delta_a^l\exp\left(-\frac{a}{2}(x-b)^2\right)~\diff x\\
&= b^l\cdot \int_\rset x^i\cdot\exp\left(-\frac{a}{2}(x-b)^2\right)~\diff x,
\end{align*}
i.e., $\Delta_a^l t_a(b) = b^l t_a(b)$ and $\fM\Delta_a^l t_a(b) = b^l \fM t_a(b) = b^l s_\sA(b) = M_l s_\sA(b)$.
\end{proof}

\begin{lem}\label{lem:deltaaPropEnd}
Let $k\in\nset$ and $F(x)$ be the Gaussian mixture
\[F(x)\quad :=\quad \sum_{i=1}^k c_i\cdot\exp\left(-\frac{a}{2}(x-b_i)^2\right)\]
for $b_1,\dots,b_k\in\rset$ pairwise different and $c_1,\dots,c_k\in\rset\setminus\{0\}$. Let
\[s_i\quad :=\quad \int_\rset x^i\cdot F(x)~\diff\lambda(x) \qquad\text{with}\qquad i=0,\dots,2k-2\]
be the moments of $F(x)$ up to degree $2k-2$. The following matrix has full rank:
\[(s,\Delta_a s,\dots,\Delta_a^{k-1} s)_{k-1}\quad =\quad
\begin{pmatrix}
s_0 & \Delta_a s_0 & \cdots & \Delta_a^{k-1} s_0\\
s_1 & \Delta_a s_1 & \cdots & \Delta_a^{k-1} s_1\\
\vdots & \vdots & & \vdots\\
s_{k-1} & \Delta_a s_{k-1} & \cdots & \Delta_a^{k-1} s_{k-1}
\end{pmatrix}.\]
\end{lem}
\begin{proof}
Take $\fM\in\rset^{k\times k}$ from \Cref{lem:transform} and set $\tilde{s}:=\fM (s_0,\dots,s_{k-1})$. Then
\begin{align*}
(s,\Delta_a s,\dots,\Delta_a^{k-1} s)_{k-1}
&= \fM^{-1} \fM (s,\Delta_a s,\dots,\Delta_a^{k-1}s)_{k-1}\\
&= \fM^{-1}(\tilde{s}, M_1\tilde{s},\dots,M_{k-1}\tilde{s})_{k-1}
\end{align*}
is full rank as in the one-dimensional case.
\end{proof}

With these properties of $\Delta_a$ we can characterize moments sequences which are represented by (\ref{eq:FxEquala}) and determine the parameters $b_i$ if $a>0$ is known.

\begin{thm}\label{thm:knowna}
Let $k,d\in\nset$ with $d\geq k$, $s=(s_0,s_1,\dots,s_d)\in\rset^{d+1}$. The following are equivalent:
\begin{enumerate}[i)]
\item For $a>0$, $c_1,\dots,c_k\in\rset\setminus\{0\}$, and $b_1,\dots,b_k\in\rset$ pairwise different we have that $s=(s_0,\dots,s_d)$ has the representing measure $\Lambda_F$ with
\begin{equation}\label{eq:FxEquala}
F(x)\quad =\quad \sum_{i=1}^k c_i\cdot \exp\left(-\frac{a}{2}(x-b_i)^2\right).
\end{equation}

\item For $a>0$ and $b_1,\dots,b_k\in\rset$ pairwise different we have that
\[(\sigma_k(b_1,\dots,b_k),\dots,\sigma_1(b_1,\dots,b_k),1)\quad\in\quad\ker(s,\Delta_a s, \Delta_a^2 s,\dots,\Delta_a^k s)_{d-k}.\]
\end{enumerate}
If additionally $d\geq 2k$, then both are equivalent to the following:
\begin{enumerate}[i)]\setcounter{enumi}{2}
\item For $a>0$ and $b_1,\dots,b_k\in\rset$ pairwise different we have that
\begin{equation*}
\ker(s,\Delta_a s, \Delta_a^2 s,\dots,\Delta_a^k s)_{d-k}\; =\; ((-1)^k\sigma_k(b_1,\dots,b_k),\dots,-\sigma_1(b_1,\dots,b_k),1)\cdot\rset.
\end{equation*}

\item For $a>0$ and $b_1,\dots,b_k\in\rset$ pairwise different we have that
\begin{equation}\label{eq:kernelKnowna}
\ker(s,\Delta_a s, \Delta_a^2 s,\dots,\Delta_a^k s)_{d-k}\quad =\quad (v_k,v_{k-1},\dots,v_1,1)\cdot\rset.
\end{equation}
and $\cZ(p) = \{b_1,\dots,b_k\}$ for
\begin{equation}\label{eq:vietaPolyKnowna}
p(x)\quad =\quad x^k + v_1 x^{k-1} + v_2 x^{k-2} + \dots + v_k.
\end{equation}
\end{enumerate}
If one of the equivalent statements (i)--(iv) and $\cH(\fM(a) s)\succeq 0$ hold, then $c_i>0$.
\end{thm}
\begin{proof}
Using $\fM(a)$ from \Cref{lem:transform} transforms each statement (i)--(iv) into the corresponding one-dimensional statements for Dirac measures (i')--(iv'). Then the equivalence of all statements (i)--(iv) follows from the equivalence of (i')--(iv').
\end{proof}

\begin{rem}
From the proof it is evident that by a coordinate change induced by $\fM(a)$ from \Cref{lem:transform} the one-dimensional case of Gaussian mixtures with the same known variance is the same as the one-dimensional case of Dirac measures. This can also be seen from $\Delta_a^l\xrightarrow{a\to\infty} M_l$.
\end{rem}

So the highly non-linear problem of finding $k$ and $b_1,\dots,b_k$ from the moments $s$ reduces to the linear problem of calculating the kernel of (\ref{eq:kernelKnowna}) and the well-studied problem of finding all roots of a univariate polynomial (\ref{eq:vietaPolyKnowna}). The coefficients $c_1,\dots,c_k$ can then be determined by linear algebra.

But \Cref{thm:knowna} only applies if we know $a$ beforehand. We therefore have to determine $a>0$ from $s=(s_0,\dots,s_d)$ as well. Set $\tilde{\Delta}_a := (\partial + aM_1)$, i.e., $\tilde{\Delta}_a = a\cdot \Delta_a$, and observe
\[\tilde{\Delta}_a^l f\quad =\quad (\partial + ax)^l f\quad =\quad \sum_{i=0}^l \binom{l}{i} a^i x^i \partial^{l-i} f\quad +\quad \sum_{\substack{i,j\geq 0:\\ i+j\leq l-1}} \alpha_{i,j} x^i \partial^j f\]
holds for some $\alpha_{i,j}\in\rset$ and all $f\in C^l(\rset,\rset)$. Applying this to (\ref{eq:FxEquala}), i.e., $f = F$, shows that the linear dependence
\[0\quad =\quad \sum_{i=0}^k v_{k-i}\cdot \Delta_a^i F \quad =\quad \sum_{i=0}^k (-1)^{k-i} \sigma_{k-i}(b_1,\dots,b_k)\cdot \Delta_a^i F\]
from \Cref{lem:vieta}, resp.\ \Cref{thm:knowna}, implies the linear dependence of $\{x^i \partial^j F \,|\, 0\leq i,j$ and $i+j\leq k\}$,
\[0\quad =\quad \sum_{i=0}^k \binom{k}{i} a^i x^i \partial^{k-i} F\quad +\quad \sum_{\substack{i,j\geq 0:\\ i+j\leq k-1}} \beta_{i,j} x^i \partial^j F\]
for some $\beta_{i,j}\in\rset$, and therefore also the moments $\{M_i \partial^j s \,|\, 0\leq i,j$ and $i+j\leq k\}$,
\[0\quad =\quad \sum_{i=0}^k \binom{k}{i} a^i M_i \partial^{k-i} s\quad +\quad \sum_{\substack{i,j\geq 0:\\ i+j\leq k-1}} \beta_{i,j} M_i \partial^j s.\]
Let us have a look at a small example.

\begin{exm}
For $k=2$ in (\ref{eq:FxEquala}) we have
\begin{align*}
0 &= a(1+a b_1 b_2) F - a^2 (b_1 + b_2) x F - a(b_1+b_2) \partial F + a^2 x^2 F + 2a x\partial F + \partial^2 F,
\end{align*}
i.e., the matrix
\begin{multline*}
(s,M_1 s, \partial s, M_2 s, M_1\partial s, \partial^2 s)_{l\geq 4} =\\
\begin{pmatrix}
s_0 & s_1 & 0 & s_2 & 0 & 0\\
s_1 & s_2 & -s_0 & s_3 & -s_1 & 0\\
s_2 & s_3 & -2s_1 & s_4 & -2s_2 & 2s_0\\
s_3 & s_4 & -3s_2 & s_5 & -3s_3 & 6s_1\\
\vdots & \vdots \\
s_l & s_{l+1} & -l s_{l-1} & s_{l+2} & -l s_l & l(l-1) s_{l-2}
\end{pmatrix}
\end{multline*}
contains the following vector in its kernel:
\[\begin{pmatrix} v_5\\ v_4\\ v_3\\ v_2\\ v_1\\ 1\end{pmatrix} \quad=\quad \begin{pmatrix}a(1+ab_1b_2)\\ -a^2(b_1+b_2)\\ -a(b_1+b_2)\\ a^2\\ 2a\\ 1\end{pmatrix}.\]
For sufficiently large $d\in\nset$ the kernel is one-dimensional. Hence,
\[a = \frac{v_1}{2}, \qquad \sigma_1 := b_1+b_2 = -2\frac{v_3}{v_1} \qquad\text{and}\qquad \sigma_2 := b_1 b_2 = \frac{2v_5 - v_1}{2v_2}\]
and by Vieta's formulas (\Cref{lem:vieta}) we have that $b_1$ and $b_2$ are the zeros of
\[p(\lambda) = \lambda^2 - \sigma_1\lambda + \sigma_2.\]
\end{exm}

The previous example provides one way to find $a>0$. It determines $a$ uniquely (and the $b_1,\dots,b_k$ simultaneously) but with the cost that more moments are required than in \Cref{thm:knowna}. In \Cref{thm:knowna} we need $2k$ moments, while for the generalized method of the previous example the matrix must be of size $\frac{k(k+1)}{2}\times K$ with $K\geq \frac{k(k+1)}{2}-1$. Hence, moments of degree at least $\frac{k^2+3k-2}{2}$ are required since the last line contains $M_k s_K=s_{K+k}$.

However, with the following approach we also get $a$ from \Cref{thm:knowna}.

\begin{dfn}\label{dfn:pa}
Let $k\in\nset$ and $s=(s_0,\dots,s_{2k})\in\rset^{2k+1}$. We define
\[\fp_s(a)\quad :=\quad a^{\frac{k(k+1)}{2}}\cdot\det((s,\Delta_a s,\dots,\Delta_a^k s)_k).\]
\end{dfn}

\begin{exm}\
\begin{enumerate}[a)]
\item For $s=(s_0,s_1,s_2)\in\rset^3$, i.e., $k=1$, we have
\[\fp_s(a) = a\cdot\det((s,\Delta_a s)_1) = a\cdot\begin{vmatrix}s_0 & s_1\\ s_1 & s_2 - \frac{1}{a}s_0\end{vmatrix} = a(s_0 s_2 - s_1^2) - s_0^2.\]

\item For $s=(s_0,s_1,s_2,s_3,s_4)\in\rset^3$, i.e., $k=2$, we have
\begin{align*}
\fp_s(a) &= a^3\cdot\det((s,\Delta_a s,\Delta_a^2 s)_2) = \begin{vmatrix}
s_0 & a s_1 & a^2 s_2 - a s_0\\
s_1 & a s_2 - s_0 & a^2 s_3 - 3a s_1\\
s_2 & a s_3 - s_1 & a^2 s_4 - 5a s_2 + 2 s_0\end{vmatrix}\\
&= a^3(s_0 s_2 s_4 -s_0 s_3^2 +2s_1 s_2 s_3 - s_1^2 s_4 - s_2^3)\\
&\quad + a^2(-s_0^2 s_4 + 3 s_0 s_1 s_3 - 3 s_0 s_2^2 + s_1^2 s_2) + a(6 s_0^2 s_2 - 4 s_0 s_1^2) - 2 s_0^3.
\end{align*}
\end{enumerate}
\end{exm}

\begin{lem}
Let $k\in\nset$ and $s=(s_0,\dots,s_{2k})\in\rset^{2k+1}$. The following holds:
\begin{enumerate}[i)]
\item $\fp_s(x)\in\rset[x]_{\leq\frac{k(k+1)}{2}}$.

\item If $s$ is represented by $\Lambda_F$ with
\[F(x)\quad =\quad \sum_{i=1}^k c_i\cdot\exp\left(-\frac{a}{2}(x-b_i)^2\right)\]
for some $a>0$, $c_1,\dots,c_k\in\rset\setminus\{0\}$, and $b_1,\dots,b_k\in\rset$ pairwise different. Then
\[\fp_s(a)\; =\; 0.\]
\end{enumerate}
\end{lem}
\begin{proof}
This follows immediately from \Cref{dfn:pa} and \Cref{thm:knowna}.
\end{proof}

The previous lemma combined with \Cref{thm:knowna} provides the following algorithm to determine a Gaussian mixture representation of $s$ with equal variance for each Gaussian component.

\begin{alg}\label{alg:eachatest}\
\begin{enumerate}[Step 1:]
\item[\textbf{Input:}] $k\in\nset$ and $s=(s_0,s_1,\dots,s_d)\in\rset^{d+1}$ with $d\geq 2k$.

\item \begin{enumerate}[a)]
\item Calculate $\fp_s(a) := a^{\frac{k(k+1)}{2}}\cdot\det((s,\Delta_a s,\dots,\Delta_a^k s)_k)$.

\item Calculate $Z:=\cZ(\fp_s)\cap\rset_{>0} = \{a_1,\dots,a_l\}$.\\ If $Z$ is empty, $s$ has no $k$-Gaussian mixtures with equal variance.
\end{enumerate}

\item For $i=1,\dots,l$:
\begin{enumerate}[a)]
\item Calculate $v_1,\dots,v_k\in\rset$ from (\ref{eq:kernelKnowna}):
\[\ker(s,\Delta_{a_i} s, \Delta_{a_i}^2 s,\dots,\Delta_{a_i}^k s)_{d-k}=(v_k,v_{k-1},\dots,v_1,1)\cdot\rset.\tag{$*$}\]
If ($*$) does not hold: $a_i$ is not a variance for $s$. Goto $i+1$.

\item Calculate zeros $\cZ(p) = \{b_1,\dots,b_k\}$ of (\ref{eq:vietaPolyKnowna}):
\[p(x) = x^k + v_1 x^{k-1} + v_2 x^{k-2} + \dots + v_k.\]
\end{enumerate}

If $p$ has complex solutions: $a_i$ is not a variance for $s$. Goto $i+1$.

\item Calculate $c_1,\dots,c_k\in\rset$ from the $t_{a_i}(b_j)$'s in (\ref{eq:tabDef}):
\[s = \sum_{j=1}^k c_j\cdot t_{a_i}(b_j).\]

\item[\textbf{Out}:] $a>0$, $b_1,\dots,b_k\in\rset$, and $c_1,\dots,c_k\in\rset$.
\end{enumerate}
\end{alg}

This algorithm can of course be modified to determine $k$ as well. Add an outer loop testing \Cref{alg:eachatest} for $k=1,\dots,\lceil\frac{d}{2}\rceil$.

\subsection*{Multiple components in dimension one.}

Now we want to investigate the one-dimensional case with $a_1,\dots,a_k>0$ arbitrary (e.g., pairwise different). For $k=2$ we have the problem already considered by Pearson \cite{pearson94}.

\begin{exm}
Let $k=2$, $a_1,a_2>0$ with $a_1\neq a_2$ and $b_1,b_2\in\rset$. For
\[F(x)\quad :=\quad c_1\cdot \exp\left(-\frac{a}{2}(x-b_1)^2 \right)\quad +\quad c_2\cdot \exp\left(-\frac{a}{2}(x-b_2)^2 \right)\]
we have that $\{F(x), xF(x), \partial F(x), x^2 F(x), x\partial F(x), \partial^2 F(x)\}$ are linearly independent. But adding $\{x^3 F(x), x^2 \partial F(x), x\partial^2 F(x)\}$ (without $\partial^3 F(x)$) makes the system linearly dependent:
\[0\quad =\quad (v_9 + v_8 x + v_7 \partial + v_6 x^2 + v_5 x\partial + v_4 \partial^2 + v_3 x^3 + v_2 x^2\partial \partial + v_1 x\partial^2) F(x).\]
We have a one-dimensional solution set spanned by
\[v_1 = 1,\quad v_2 = a_1 + a_2,\quad v_3 = a_1 a_2,\quad\text{and}\quad v_{i:4\leq i\leq 9}\in\qset(a_1,a_2,b_1,b_2).\]
So $a_1$ and $a_2$ are the zeros of
\[p(x) = x^2 + v_2 x + v_3\]
by the Vieta's Formulas (\Cref{lem:vieta}).
\end{exm}

One might to be seduced by this example and the opinion that by replacing the restriction $a_1 = \dots = a_k = a$ by arbitrary $a_i > 0$ that less Gaussian distribution are required. But \Cref{thm:mixtureBoundBelow} shows that there are moment sequences with very large mixture Carath\'eodory numbers.

\subsection*{Multi-dimensional Gaussian mixtures.}
So far we only dealt with the one-dimensional case of Gaussian mixture reconstruction from moments. And this was even done with the restriction $a_1=\dots=a_k=a>0$. In \cite{didio19HilbertArxiv} we proved new lower bounds for the Carath\'eodory numbers for Dirac measures which grow asymptotically close to the Richter upper bound. Now we show that for Gaussian mixtures the same lower bounds hold even when arbitrary variances are allowed.

Before we can state our last main theorem, we need the following definition.

\begin{dfn}\label{dfn:highestOrder}
Let $\cA$ be a finite-dimensional vector space of measurable functions on a measurable space $(\cX,\fA)$ and $\delta_{\sigma,\xi}$ probability measures as in \Cref{dfn:prob}. A function $a\in\cA$ is called \emph{non-negative of highest order} (with respect to the measures $\delta_{\sigma,\xi}$) if $a\geq 0$ and for any sequence $(c_i,\sigma_i,\xi_i)_{i\in\nset}\subseteq\rset_{\geq 0}\Sigma\times\cX$ with
\[\int_\cX a(x)~\diff(c_i\cdot\delta_{\sigma_i,\xi_i})(x)\quad \xrightarrow{i\to\infty}\quad 0\]
there exists a subsequence $(i_j)_{j\in\nset}$ with one of the following properties:
\begin{enumerate}[i)]
\item $\sigma_{i_j}\xrightarrow{j\to\infty} \sigma_0$ and $\xi_{i_j}\xrightarrow{j\to\infty} \xi\in\cZ(a)$, or

\item $\int_\cX b(x)~\diff(c_{i_j}\cdot\delta_{i_j})(x) \xrightarrow{j\to\infty} 0$ for all $b\in\cA$.
\end{enumerate}
\end{dfn}

Note, being of highest order depends in general on the measures $\delta_ {\sigma,\xi}$. The following are examples for non-negative polynomials of highest order.

\begin{exm}\label{exm:highestOrder}
Let $d,n\in\nset$, $\cX=\rset^n$, and $\cA=\rset[x_1,\dots,x_n]_{\leq 2d}$. Let $p\in\cA$ be non-negative, with finitely many zeros and without zeros at infinity (its homogenization has no zeros with $x_0=0$). Then $p$ is non-negative of highest order with respect to Gaussian or log-normal measures. In particular
\[\sum_{i=1}^n (x_i-1)^2\cdots (x_i-d)^2\]
is non-negative of highest order.
\end{exm}

Recall from \cite{didio17Cara} that with $\cX\subseteq\rset^n$ open and $\cA$ a finite-dimensional space of differentiable functions on $\cX$, then $\cN_\sA\in\nset$ is the smallest $k\in\nset$ such that $DS_{k,\sA}(C,X)$ has full rank for some $(C,X)\in \rset_{\geq 0}^k\times\cX^k$ where
\[S_{k,\sA}: \rset_{\geq 0}^k\times\cX^k\to\rset^{\dim\cA},\ (C,X)\mapsto \sum_{i=1}^k c_i\cdot s_\sA(x_i)\]
with $C=(c_1,\dots,c_k)$ and $X=(x_1,\dots,x_k)$.

\begin{thm}\label{thm:mixtureBoundBelow}
Let $(\cX,\fA)$ be a measurable space, $\cA$ be a finite-dimensional space of measurable functions on $(\cX,\fA)$ with an $e\in\cA$ such that $e\geq 1$ on $\cX$, and $\delta_{\sigma,\xi}$ probability measures on $\cX$ as in \Cref{dfn:prob}. Let $a\in\cA$ be non-negative of highest order with finitely many zeros $\cZ(a) = \{x_1,\dots,x_k\}$. Then there exists a moment sequence $s\in\inter\cS_\sA = \inter\cT_\sA$ with
\begin{equation}\label{eq:mixtureBoundZeros}
\cat_\sA^M(s) \quad =\quad \dim\lin\{s_\sA(x_i)\,|\, i=1,\dots,k\}.
\end{equation}
If additionally $\cX\subseteq\rset^n$ is open, $n\in\nset$, and $\cA$ is $r$-differentiable with $r > \dim\cA - \cN_\sA\cdot(n+1)$, then $s$ has an open neighborhood $U$ such that (\ref{eq:mixtureBoundZeros}) holds for all $s'\in U$.
\end{thm}
\begin{proof}
Let $K:=\dim\lin\{s_\sA(x_i)\,|\, i=1\dots,k\}\leq k$ and $s^{(0)} := \sum_{i=1}^k s_\sA(x_i)$. Then $L_{s^{(0)}}(a) = 0$ and by \cite[Thm.\ 18]{didio17Cara} (\Cref{thm:caraLowerZeroSet}) we have $\cat_\sA(s^{(0)})=K$.

Let $(s^{(i)})_{i\in\nset}\subset\inter\cS_\sA$ be such that $s^{(i)}\to s^{(0)}$ as $i\to\infty$. By \cite[Thm.\ 17(ii)]{didio18gaussian} (\Cref{thm:interMixtures}) any $s^{(i)}$, $i\geq 1$, has a mixture representation
\[\mu_i\quad:=\quad\sum_{j=1}^{K_i} c_{i,j}\cdot \delta_{\sigma_{i,j},\xi_{i,j}} \tag{$*$}\]
with $K_i = \cat_\sA^M(s^{(i)})\leq \dim\cA$ (i.e., $K_i$ are minimal), $c_{i,j}>0$, $\sigma_{i,j}\in\Sigma$, and $x_{i,j}\in\cX$. Since $K_i\in\nset$ we have $K_0 := \liminf_{i\to\infty} K_i$ and after choosing a subsequence of $(s^{(i)})_{i\in\nset}$ we can assume that $K_i = K_0$ for all $i\in\nset$.

Let us show that $K_0 \geq K$ holds. Since $a$ is non-negative of highest order, we can assume that the $(c_{i,j},\sigma_{i,j},\xi_{i,j})$ fulfill (i) or (ii) in \Cref{dfn:highestOrder} by taking a subsequence $(i_l)_{l\in\nset}$. By reordering the $j$'s in ($*$) we can assume that (i) holds for all $j=1,\dots,M$ and (ii) for all $j=M+1,\dots,K_0$. Since $c_{i,j}\geq 0$ and $c_{i,1}+\dots+c_{i,K_0} = L_{s^{(i)}}(e) = k$ we can assume that $c_{i,j}\xrightarrow{i\to\infty} c_j$ for all $j=1,\dots,M$. But (ii) implies
\[\int_\cX s_\sA(x)~\diff\left(\sum_{j=M+1}^{K_0} c_{i,j}\cdot \delta_{\sigma_{i,j},\xi_{i,j}}\right)(x)\quad\rightarrow\quad 0\]
and therefore we have
\[\int_\cX s_\sA(x)~\diff\left(\sum_{j=1}^{M} c_{i,j}\cdot \delta_{\sigma_{i,j},\xi_{i,j}}\right)(x)\quad\rightarrow\quad \sum_{j=1}^M c_j\cdot s_\sA(\xi_j)\quad =\quad s^{(0)},\]
i.e., $K \leq M \leq K_0$. Hence $K_0\geq K$ implies that all $s^{(i)}$ fulfill $\cat_\sA^M(s^{(i)}) = K_0 \geq K$.

If $\cA$ are $r$-differentiable functions, then the sequence $(s^{(i)})_{i\in\nset}$ can be chosen to contain only regular moment sequences by Sard's Theorem \cite{sard42} (see \cite{didio17Cara}). Hence, for each $i\geq N$ there is an open neighborhood $U_i$ of $s^{(i)}$ such that all $s'\in U_i$ fulfill $\cat_\sA^M(s^{(i)}) = \cat_\sA^M(s')$.
\end{proof}

So from the proof it is evident that the constructed $s$ with (\ref{eq:mixtureBoundZeros}) is close to the boundary of the moment cone, more precisely close to the boundary face represented by \mbox{$a\in\cA$}. \cite{didio18gaussian}, \cite{didio19HilbertArxiv}, \Cref{thm:mixtureBoundBelow}, and \Cref{exm:highestOrder} explicitly provide the following.

\begin{cor}\label{cor:oneDimGaussBounds}
Let $d\in\nset$ and $\cX=\rset$. For the one-dimensional Gaussian (and log-normal) measures with $\cA = \rset[x]_{\leq d}$ we have
\[\left\lfloor\frac{d}{2}\right\rfloor \quad\leq\quad  \cat_{\sA_{1,d}}^M \quad\leq\quad \left\lfloor\frac{d}{2}\right\rfloor + 1.\]
\end{cor}
\begin{proof}
\Cref{exm:highestOrder} and \Cref{thm:mixtureBoundBelow} gives the lower bound and \cite[Cor.\ 36]{didio18gaussian} the upper bound.
\end{proof}

\Cref{cor:oneDimGaussBounds} (and \ref{cor:lowerBoundsExplicit}) explains why we only discussed the reconstruction of one-dimen\-sional Gaussian mixtures with equal variances $a_1=\dots=a_k=a>0$ at the beginning of this section. There are moment sequences where it is sufficient to represent them by mixtures of $\delta_{\sigma_i,\xi_i}$ with $\sigma_1=\dots=\sigma_k$ and relaxation of this restriction does not improve the required number of components. Especially in higher dimensions we will see that even in the case of $\cA = \rset[x_1,\dots,x_n]_{\leq 2d}$ with Gaussian measures the number of components becomes very large, close to $\dim\cA = \binom{n+2d}{n}$, see \Cref{cor:lowerBoundsAsym}.

\begin{exm}[$\cX=\rset^2$]
Let $\cX=\rset^2$ and $\cA=\rset[x_1,x_2]_{\leq d}$, $d\in 2\nset$. By a rotation of $\pset^n$ we can assume that for homogeneous polynomials in $\rset[x_0,x_1,x_2]_{= d}$ with finitely many zeros no zero is at infinity ($x_0=0$).
\begin{enumerate}[a)]
\item \underline{$d=4$:} The \emph{Motzkin polynomial} \cite{motzkin65} has $6$ projective zeros, is non-negative of highest order, and the point evaluations at these zeros are linearly independent, see \cite[Exm.\ 31]{didio17Cara}. So $\cat^M_{\sA_{2,4}}\geq 6$, i.e., there is a moment sequence/functional on $\rset[x_1,x_2]_{\leq 4}$ which can be represented by a sum of $6$ Gaussians but not less. The upper bound for the Dirac measures in the projective case is also $6$ \cite{reznick92}.

\item \underline{$d=6$:} The \emph{Robinson polynomial} \cite{robinson69} has $10$ projective zeros, is non-negative of highest order and all point evaluations at these zeros are also linearly independent, see \cite[p.\ 1635]{didio17Cara}. So $\cat^M_{\sA_{2,6}}\geq 10$. Note, that for Dirac measures we have the Carath\'eodory number $11$ in the projective case, see \cite{kunertPhD14}.

\item \underline{$d=10$:} The \emph{Harris polynomial} \cite{harris99} has $30$ zeros, is non-negative of highest order, and the point evaluations at these zeros are linearly independent, see \cite[Exm.\ 63]{didio17Cara}. Hence, $\cat^M_{\sA_{2,10}}\geq 30$. An upper bound for Dirac measures in the projective setting is $32$, see \cite[Exm.\ 63]{didio17Cara}.

\item \underline{$d\in 2\nset$:} In \cite[Lem.\ 8.6]{rienerOptima} it was shown that the point evaluations on the grid
\[G = \{1,2,\dots,d\}^2 = \cZ(p) \quad\text{with}\quad p(x_1,x_2)= \prod_{i=1}^d (x_1-i)^2+ \prod_{i=1}^d (x_2-i)^2\]
are linearly independent on $\cA = \rset[x_1,x_2]_{\leq 2d}$. Hence $\cat^M_{\sA_{2,2d}} \geq d^2$. Additionally, it was shown that $\cat_{\sA_{2,2d+1}}\leq \frac{3}{2}d(d+1)+1$ holds. With $\cat_{\sA_{2,2d}}\leq \cat_{\sA_{2,2d+1}}$, \Cref{thm:mixtureBoundBelow}, and \cite[Thm.\ 35]{didio18gaussian} we have
\[d^2\quad\leq\quad \cat^M_{\sA_{2,2d}} \quad\leq\quad \frac{3}{2}d(d+1)+1.\]
\end{enumerate}
\end{exm}

In \cite{didio19HilbertArxiv} the point evaluation on the grid was extended to higher dimensions and improved lower bounds where found. In fact, the following result was shown.

\begin{prop}[{\cite[Prop.\ 5.3]{didio19HilbertArxiv}}]
Let $n,d\in\nset$, $k\in\{0,1\}$, $\cX=\rset^n$, and $G=\{1,\dots,d\}^n$. Then
\[s = \sum_{x\in G} s_{\sA_{n,2d+k}}(x) \qquad\text{resp.}\qquad L = \sum_{x\in G} l_x:\rset[x_1,\dots,x_n]_{\leq 2d+k}\rightarrow\rset\]
supported on the grid $G$ with the representing measure $\mu = \sum_{x\in G} \delta_x$ has the Carath\'eodory number
\[\cat_{\sA_{n,2d+k}}(s) = \begin{cases}
\left(\begin{smallmatrix} n+2d\\ n\end{smallmatrix}\right) - n\cdot \left(\begin{smallmatrix} n+d\\ n\end{smallmatrix}\right) + \left(\begin{smallmatrix} n\\ 2\end{smallmatrix}\right) & \text{for}\ k=0,\phantom{\Big(}\\
\left(\begin{smallmatrix} n+2d+1\\ n\end{smallmatrix}\right) - n\cdot \left(\begin{smallmatrix} n+d+1\\ n\end{smallmatrix}\right) + 3\cdot \left(\begin{smallmatrix} n+1\\ 3\end{smallmatrix}\right) & \text{for}\ k=1.\phantom{\Big(} \end{cases}\]
\end{prop}

Since the grid $G=\{1,\dots,d\}^n$ is the zero set of a non-negative polynomial of highest order (\Cref{exm:highestOrder}), \Cref{thm:mixtureBoundBelow} implies the following.

\begin{cor}\label{cor:lowerBoundsExplicit}
Let $n,d\in\nset$, $\cX=\rset^n$, and $\cA = \rset[x_1,\dots,x_n]_{\leq 2d}$. Then there is a moment sequence $s\in\inter\cS_{\sA_{n,2d}}$ and an open neighborhood $U\subset\inter\cS_{\sA_{n,2d}}$ of $s$ such that
\begin{equation}\label{eq:gridLowerBound}
\cat^M_{\sA_{n,2d}}(s')\quad =\quad \binom{n+2d}{n} - n\cdot \binom{n+d}{n} + \binom{n}{2}
\end{equation}
for all $s'\in U$, i.e., every $s'\in U$ is a linear combination of (\ref{eq:gridLowerBound}) many Gaussian distributions but not less.
\end{cor}

Hence, like in \cite[Thm.\ 5.6]{didio19HilbertArxiv} we have
\begin{align*}
\liminf_{d\rightarrow\infty} \frac{\cat^M_{\sA_{n,2d}}}{|\sA_{n,2d}|}\quad &\geq\quad 1-\frac{n}{2^n} & \textrm{ for all } n &\in\nset\\
\intertext{and}
\lim_{n\rightarrow\infty} \frac{\cat^M_{\sA_{n,2d}}}{|\sA_{n,2d}|}\quad &= \quad 1 &\textrm{ for all } d&\in\nset.
\end{align*}

We end with the following asymptotic result which, as in the case of atomic measures \cite[Cor.\ 5.8]{didio19HilbertArxiv}, demonstrates that also the truncated moment problem with Gaussian mixtures is cursed by high dimensions. Note, an upper bound for the number of components is $\cat^M_{\sA_{n,2d}}\leq\binom{2d+n}{n}-1$, see \cite[Thm.\ 32]{didio18gaussian}.

\begin{cor}\label{cor:lowerBoundsAsym}
Let $d\in\nset$ and $\varepsilon > 0$. Then there is an $n\in\nset$ such that there is a moment functional $L:\rset[x_1,\dots,x_n]_{\leq 2d}\to\rset$ which can be written as a sum of
\[(1-\varepsilon)\cdot\binom{2d+n}{n}\]
Gaussian distributions but not less.
\end{cor}

\section*{Acknowledgment}

We thank Mario Kummer for the productive discussions and advise on the paper. We thank Mioara Joldes, Florent Br\'ehard, and Jean-Bernard Lasserre to provide the references \cite{lasserre08,henrion14,marx18arxiv,brehard19unpub}. We want to thank Bernard Mourrain for the fruitful discussion at the Arctic Applied Algebra conference organized by Philippe Moustrou, Verena Reichle, Cordian Riener, and Hugues Verdure in Troms{\o}, April 2019.


\begin{thebibliography}{MVKW95}

\bibitem[AFS16]{amendo16}
C.~Am\'{e}ndola, J.-C. Faug\`{e}re, and B.~Sturmfels, \emph{Moment varieties of
  gaussian mixtures}, J.~Alg.\ Stat. \textbf{7} (2016), 14--28.

\bibitem[Akh65]{akhiezClassical}
N.~I. Akhiezer, \emph{The classical moment problem and some related questions
  in analysis}, Oliver \& Boyd, Edinburgh, 1965.

\bibitem[Ana06]{anast06}
G.~A. Anastassiou, \emph{Applications of geometric moment theory related to
  optimal portfolio management}, Comput.\ Math.\ Appl. \textbf{51} (2006),
  1405--1430.

\bibitem[APST19]{ammari19}
H.~Ammari, M.~Putinar, A.~Streenkamp, and F.~Triki, \emph{Identification of an
  algebraic domain in two dimensions from a finite number of its generalized
  polarization tensors}, Math.\ Ann. (2018/19), in press,
  https://doi.org/10.1007/s00208-018-1780-y.

\bibitem[Bal61]{balins61}
M.~L. Balinski, \emph{An algorithm for finding all vertices of convex
  polyhedral sets}, J.~Soc.\ Indust.\ Appl.\ Math. \textbf{9} (1961), 72--88.

\bibitem[Bar91]{barvin91}
A.~I. Barvinok, \emph{Calculation of exponential integrals}, Zap.\ Nau\v{c}.\
  Semin.\ POMI \textbf{192} (1991), 175--176.

\bibitem[Bar92]{barvin92}
\bysame, \emph{Exponential integrals and sums over convex polyhedra}, Funkc.\
  Anal.\ Prilozh. \textbf{26} (1992), 64--66.

\bibitem[BGL07]{becker07}
B.~Beckermann, G.~H. Golub, and G.~Labahn, \emph{On the numerical condition of
  a generalized {H}ankel eigenvalue problem}, Numer.\ Math. \textbf{106}
  (2007), 41--68.

\bibitem[BJL19]{brehard19unpub}
F.~Br\'ehard, M.~Joldes, and J.-B. Lasserre, \emph{On a moment problem with
  holonomic functions}, 2019, https://hal.archives-ouvertes.fr/hal-02006645.

\bibitem[Bri88]{brion88}
M.~Brion, \emph{Points entiers dans les poly\`{e}dres convexes}, Ann.\ Sci.\
  \'Ec.\ Norm.\ Super. \textbf{21} (1988), 653--663.

\bibitem[Che93]{chen93}
C.-C. Chen, \emph{Improved moment invariants for shape discrimination}, Pattern
  Recognit. \textbf{26} (1993), 683--686.

\bibitem[DBN92]{dai92}
M.~Dai, P.~Baylou, and M.~Najim, \emph{An efficient algorithm for computation
  of shape moments from run-length codes or chain codes}, Pattern Recognit.
  \textbf{25} (1992), 1119--1128.

\bibitem[dD19]{didio18gaussian}
P.~J. di~Dio, \emph{The multidimensional truncated {M}oment {P}roblem: Gaussian
  and {L}og-{N}ormal {M}ixtures, their {C}arath\'eodory {N}umbers, and {S}et of
  {A}toms}, Proc.\ Amer.\ Math.\ Soc. \textbf{147} (2019), 3021--3038,
  arXiv:1804.07058.

\bibitem[dDK19]{didio19HilbertArxiv}
P.~J. di~Dio and M.~Kummer, \emph{The multidimensional truncated moment
  problem: Carath\'eodory {N}umbers from {H}ilbert {F}unctions},
  https://arxiv.org/abs/1903.00598v2.

\bibitem[dDS18a]{didioConeArXiv}
P.~J. di~Dio and K.~Schm\"{u}dgen, \emph{The multidimensional truncated moment
  problem: The moment cone}, https://arxiv.org/abs/1809.00584.

\bibitem[dDS18b]{didio17Cara}
P.~J. di~Dio and K.~Schm{\"u}dgen, \emph{{The} multidimensional truncated
  moment problem: {C}arath\'eodory {N}umbers}, J.~Math.\ Anal.\ Appl.
  \textbf{461} (2018), 1606--1638.

\bibitem[FN10]{fialkow10}
L.~A. Fialkow and J.~Nie., \emph{Positivity of {R}iesz functionals and
  solutions of quadratic and quartic moment problems}, J.~Funct.\ Anal.
  \textbf{258} (2010), 328--356.

\bibitem[GLPR12]{gravin12}
N.~Gravin, J.~Lasserre, D.~V. Pasechnik, and S.~Robins, \emph{The inverse
  moment problem for convex polytopes}, Discrete Comput.\ Geom. \textbf{48}
  (2012), 596--621.

\bibitem[GMV99]{golub99}
G.~H. Golub, P.~Milfar, and J.~Varah, \emph{A stable numberical method for
  inverting shape from moments}, SIAM J.\ Sci.\ Comput. \textbf{21} (1999),
  no.~4, 1222--1243.

\bibitem[GNPR14]{gravin14}
N.~Gravin, D.~Nguyen, D.~V. Pasechnik, and S.~Robins, \emph{The inverse moment
  problem for convex polytopes: Implementation aspects}, arXiv:1409.3130v2.

\bibitem[GPSS18]{gravin18}
N.~Gravin, D.~Pasechnik, B.~Shapiro, and M.~Shapiro, \emph{On moments of a
  polytope}, Anal.\ Math.\ Phys. \textbf{8} (2018), 255--287.

\bibitem[Gru09]{grubbDistributions}
G.~Grubb, \emph{Distributions and operators}, Spinger, New York, 2009.

\bibitem[Har99]{harris99}
W.~R. Harris, \emph{Real {E}ven {S}ymmetric {T}ernary {F}orms}, J.~Alg.
  \textbf{222} (1999), 204--245.

\bibitem[HK14]{henrion14}
D.~Henrion and M.~Korda, \emph{Convex computation of the region of attraction
  of polynomial control systems}, IEEE Trans.\ Aut.\ Control \textbf{59}
  (2014), 297--312.

\bibitem[Hu62]{hu62}
M.-K. Hu, \emph{Visual pattern recognition by moment invariants}, IRE Trans.\
  Inf.\ Theory \textbf{12} (1962), 179--187.

\bibitem[Kem68]{kemper68}
J.~H.~B. Kemperman, \emph{The {G}eneral {M}oment {P}roblem, a {G}eometric
  {A}pproach}, Ann.\ Math.\ Stat. \textbf{39} (1968), 93--122.

\bibitem[Kem87]{kemper87}
\bysame, \emph{Geometriy of the moment problem}, Proc.\ Sym.\ Appl.\ Math.
  \textbf{37} (1987), 16--53.

\bibitem[KN77]{kreinMarkovMomentProblem}
M.~G. Kre\u{\i}n and A.~A. Nudel'man, \emph{The {M}arkow {M}oment {P}roblem and
  {E}xtremal {P}roblems}, American Mathematical Society, Providence, Rhode
  Island, 1977.

\bibitem[KSS18]{kohn18}
K.~Kohn, B.~Shapiro, and B.~Sturmfels, \emph{Moment varieties of measures on
  polytopes}, arXiv:1807.10258v1.

\bibitem[Kun14]{kunertPhD14}
A.~Kunert, \emph{Facial {S}tructure of {C}ones of non-negative {F}orms}, Ph.D.
  thesis, Universtity of Konstanz, 2014.

\bibitem[Lan80]{landauMomAMSProc}
H.~J. Landau (ed.), \emph{Moments in {M}athematics}, Proceedings of Symposia in
  applied Mathematics, vol.~37, Providence, RI, American Mathematical Society,
  1980.

\bibitem[Las15]{lasserreSemiAlgOpt}
J.~B. Lasserre, \emph{An introduction to polynomial and semi-algebraic
  optimization}, Cambridge University Press, Cambridge, 2015.

\bibitem[Lau09]{lauren09}
M.~Laurent, \emph{Sums of {S}quares, {M}oment {M}atrices and {P}olynomial over
  {O}ptimization}, Emerging application of algebraic geometry, IMA Vol. Math.
  Appl., vol. 149, Springer, New York, 2009, pp.~157--270.

\bibitem[Law91]{lawrence91}
J.~Lawrence, \emph{Polytope volume computation}, Math.\ Comput. \textbf{57}
  (1991), 259--271.

\bibitem[LPHT08]{lasserre08}
J.-B. Lasserre, C.~Prieur, D.~Henrion, and E.~Tr\'elat, \emph{Nonlinear optimal
  control via occupation measures and {LMI}-relaxations}, SIAM J.\ Control
  Optim. \textbf{47} (2008), 1649--1666.

\bibitem[LR82]{lee82}
Y.~T. Lee and A.~A.~G. Requicha, \emph{Algorithms for computing the volume and
  other integral properties of solids. {I}. known methods and open issues},
  Comm.\ ACM \textbf{25} (1982), 635--641.

\bibitem[Mar08]{marshallPosPoly}
M.~Marshall, \emph{Positive {P}olynomials and {S}ums of {S}quares},
  Mathematical Surveys and Monographs, no. 146, American Mathematical Society,
  Rhode Island, 2008.

\bibitem[MMR05]{martin05}
J.-M. Martin, K.~Mengersen, and C.~P. Robert, \emph{Bayesian modelling and
  inference on mixtures of distributions}, Handbook of Statistics \textbf{25}
  (2005), 459--507.

\bibitem[MN68]{manas68}
M.~Ma\v{n}as and J.~Nedoma, \emph{Finding all vertices of a convex polyhedron},
  Numer.\ Math. \textbf{12} (1968), 226--229.

\bibitem[Mot67]{motzkin65}
T.~S. Motzkin, \emph{The arithmetic-geometric inequality}, Inequalities (New
  York) (O.~Shisha, ed.), Proc. of Sympos. at Wright-Patterson AFB, August
  19--27, 1965, Academic Press, 1967, pp.~205--224.

\bibitem[MR80]{mathei80}
T.~H. Matheiss and D.~S. Rubin, \emph{A survey and comparison of methods for
  finding all vertices of convex polyhedral sets}, Math.\ Oper.\ Res.
  \textbf{5} (1980), 167--185.

\bibitem[M{\"u}n14]{muntz14}
C.~H. M{\"u}ntz, \emph{Mathematische {A}bhandlungen {H}ermann {A}mandus
  {S}chwarz zu seinem f\"{u}nfzigj\"ahrigen {D}oktorjubil\"{a}um am 6.\
  {A}ugust 1914 gewidmet von {F}reunden und {S}ch\"{u}lern.}, ch.~\"Uber den
  {A}pproximationssatz von {W}eierstrass, pp.~303--312, Springer, Berlin, 1914.

\bibitem[MVKW95]{milanf95}
P.~Milanfar, G.~Verghese, W.~Karl, and A.~Willsky, \emph{Reconstructing
  polygons from moments with connections to array processing}, IEEE Trans.\
  Signal Proc. \textbf{43} (1995), 432--443.

\bibitem[MWHL18]{marx18arxiv}
S.~Marx, T.~Weisser, D.~Henrion, and J.-B. Lasserre, \emph{A moment approach
  for entropy solutions to nonlinear hyperbolic {PDE}s}, arXiv:1807.02306v1.

\bibitem[Pea94]{pearson94}
K.~Pearson, \emph{Contributions to the mathematical theory of evolution},
  Phil.\ Trans.\ Roy.\ Soc.\ London~A \textbf{185} (1894), 71--110.

\bibitem[Rez92]{reznick92}
B.~Reznick, \emph{Sums of even powers of real linear forms}, Mem. Amer. Math.
  Soc. \textbf{96} (1992),  Amer. Math. Soc., Providence, R.I. 1992.

\bibitem[PK92]{pukhli92}
A.~V. Pukhlikov and A.~G. Khovanskii, \emph{The {R}iemann--{R}och theorem for
  integrals and sums of quasipolynomials on virtual polytopes}, Algebra Anal.
  \textbf{4} (1992), 188--216.

\bibitem[Ric57]{richte57}
H.~Richter, \emph{Parameterfreie {A}bsch\"atzung und {R}ealisierung von
  {E}rwartungswerten}, Bl.\ Deutsch.\ Ges.\ Versicherungsmath. \textbf{3}
  (1957), 147--161.

\bibitem[Rob69]{robinson69}
R.~M. Robinson, \emph{Some definite polynomials which are not sums of squares
  of real polynomials}, Notices Amer.\ Math.\ Soc. \textbf{16} (1969), 554.

\bibitem[Rog58]{rogosi58}
W.~W. Rogosinski, \emph{Moments of non-negative mass}, Proc.\ R.\ Soc.\ Lond.\
  A \textbf{245} (1958), 1--27.

\bibitem[Ros52]{rosenb52}
P.~C. Rosenbloom, \emph{Quelques classes de probl\`{e}me extr\'{e}maux. {II}},
  Bull.\ Soc.\ Math.\ France \textbf{80} (1952), 183--215.

\bibitem[RS18]{rienerOptima}
C.~Riener and M.~Schweighofer, \emph{Optimization approaches to quadrature: new
  characterizations of {G}aussian quadrature on the line and quadrature with
  few nodes on plane algebraic curves, on the plane and in higher dimensions},
  J.~Compl. \textbf{45} (2018), 22--54.

\bibitem[Sar42]{sard42}
A.~Sard, \emph{The measure of critical values of differentiable maps}, Bull.\
  Amer.\ Math.\ Soc. \textbf{48} (1942), 883--890.

\bibitem[Sch17]{schmudMomentBook}
K.~Schm\"{u}dgen, \emph{The {M}oment {P}roblem}, Springer, New York, 2017.

\bibitem[SMD{\etalchar{+}}07]{sommer07}
I.~Sommer, O.~M\"uller, F.~S. Domingues, O.~Sander, J.~Weickert, and
  T.~Lengauer, \emph{Moment invariants as shape recognition technique for
  comparing protein binding sites}, Bioinformatics \textbf{23} (2007),
  3139--3146.

\bibitem[ST43]{shohat43}
J.~A. Shohat and J.~D. Tamarkin, \emph{The {P}roblem of {M}oments}, Amer.\
  Math.\ Soc., Providence, R.I., 1943.

\bibitem[Sti94]{stielt94}
T.~J. Stieltjes, \emph{Recherches sur les fractions continues}, Ann.~Fac.\ Sci.
  Toulouse \textbf{8} (1894), no.~4, J1--J122.

\bibitem[Sto16]{stoyan16}
J.~Stoyanov, \emph{Moment properties of probability distributions used in
  stochastic financial models}, Recent Advances in Financial Engineering 2014
  Proceedings of the TMU Finance Workshop 2014, World Scientific Publishing Co.
  Pte. Ltd., 2016, pp.~1--27.

\bibitem[Sz{\'a}16]{szasz16}
O.~Sz{\'a}sz, \emph{\"{U}ber die {A}pproximation stetiger {F}unktionen durch
  lineare {A}ggregate von {P}otenzen}, Math.\ Ann. \textbf{77} (1916),
  482--496.

\bibitem[TSM85]{titter85}
D.~M. Titterington, A.~F.~M. Smith, and U.~E. Makov, \emph{{S}tatistical
  {A}nalysis of {F}inite {M}ixture {D}istributions}, John Wiley \& Son,
  Chichester, 1985.

\end{thebibliography}

\newcommand{\etalchar}[1]{$^{#1}$}
\providecommand{\href}[2]{#2}

\end{document}